\documentclass[preprint]{imsart}

\usepackage[english]{babel}
\usepackage[T1]{fontenc}
\usepackage[utf8]{inputenc}
\usepackage{lmodern}
\usepackage{calc}

\usepackage{amsmath}
\usepackage{amsfonts}
\usepackage{amssymb}
\usepackage{amsthm}
\usepackage{euscript}
\usepackage{xspace}
\usepackage{bm}
\usepackage{enumerate}
\usepackage{hyperref}
\usepackage{graphicx}
\usepackage{subcaption}

\usepackage[numbers]{natbib}

\startlocaldefs

\newtheorem{definition}{Definition}[section]

\newtheorem{lemma}{Lemma}[section]
\newtheorem{theorem}{Theorem}[section]
\newtheorem{proposition}{Proposition}[section]

\theoremstyle{definition}
\newtheorem{example}{Example}

\numberwithin{equation}{section}

\DeclareMathOperator{\sign}{sign}
\DeclareMathOperator{\gr}{Gr}
\DeclareMathOperator{\diam}{diam}

\newcommand{\cadlag}{càdlàg\xspace}

\newcommand{\Bi}{\ensuremath{\mathcal B}\xspace}

\newcommand{\Ei}{\ensuremath{\mathcal E}\xspace}
\newcommand{\Fi}{\ensuremath{\mathcal F}\xspace}

\newcommand{\Hi}{\ensuremath{\mathcal H}\xspace}

\newcommand{\Si}{\ensuremath{\mathcal S}\xspace}

\newcommand{\dimU}{\ensuremath{\dim_{*}}\xspace}
\newcommand{\dimH}{\ensuremath{{\dim}_{\text{\normalfont\scriptsize H}}}\xspace}

\newcommand{\dimBl}{\ensuremath{\underline{\dim}_{\text{\normalfont\scriptsize B}}}\xspace}
\newcommand{\dimBu}{\ensuremath{\overline{\dim}_{\text{\normalfont\scriptsize B}}}\xspace}
\newcommand{\dimB}{\ensuremath{{\dim}_{\text{\normalfont\scriptsize B}}}\xspace}

\newcommand{\dimUt}[1]{\ensuremath{\dim_{*,#1}}\xspace}
\newcommand{\dimHt}[1]{\ensuremath{{\dim}_{\text{\normalfont\scriptsize H},#1}}\xspace}

\newcommand{\dimBlt}[1]{\ensuremath{\underline{\dim}_{\text{\normalfont\scriptsize B},#1}}\xspace}
\newcommand{\dimBut}[1]{\ensuremath{\overline{\dim}_{\text{\normalfont\scriptsize B},#1}}\xspace}
\newcommand{\dimBt}[1]{\ensuremath{{\dim}_{\text{\normalfont\scriptsize B},#1}}\xspace}

\newcommand{\pth}[1]{(#1)}
\newcommand{\pthb}[1]{\bigl(#1\bigr)}
\newcommand{\pthB}[1]{\Bigl(#1\Bigr)}
\newcommand{\pthbb}[1]{\biggl(#1\biggr)}

\newcommand{\bkt}[1]{[#1]}
\newcommand{\bktb}[1]{\bigl[#1\bigr]}
\newcommand{\bktB}[1]{\Bigl[#1\Bigr]}
\newcommand{\bktbb}[1]{\biggl[#1\biggr]}
\newcommand{\bktBB}[1]{\Biggl[#1\Biggr]}

\newcommand{\brc}[1]{\{#1\}}
\newcommand{\brcb}[1]{\bigl\{#1\bigr\}}
\newcommand{\brcB}[1]{\Bigl\{#1\Bigr\}}
\newcommand{\brcbb}[1]{\biggl\{#1\biggr\}}

\newcommand{\dt}{\ensuremath{\mathrm d}\xspace} 
\newcommand{\eqdef}{:=}

\newcommand{\ivoo}[1]{\ensuremath{(#1)}}
\newcommand{\ivoob}[1]{\ensuremath{\bigl(#1\bigr)}}

\newcommand{\ivof}[1]{\ensuremath{(#1]}}
\newcommand{\ivofb}[1]{\ensuremath{\bigl(#1\bigr]}}

\newcommand{\ivfo}[1]{\ensuremath{[#1)}}

\newcommand{\ivff}[1]{\ensuremath{[#1]}}

\newcommand{\abs}[1]{\lvert#1\rvert}
\newcommand{\absb}[1]{\bigl\lvert#1\bigr\rvert}

\newcommand{\absbb}[1]{\biggl\lvert#1\biggr\rvert}

\newcommand{\norm}[1]{\lVert#1\rVert}

\newcommand{\ceil}[1]{\lceil#1\rceil}

\newcommand{\esp}[2][]{\mathbb{E}#1\bkt{#2}}
\newcommand{\espb}[2][]{\mathbb{E}#1\bktb{\hspace{1pt}#2\hspace{1pt}}}
\newcommand{\espB}[2][]{\mathbb{E}#1\bktB{#2}}

\newcommand{\espBB}[2][]{\mathbb{E}#1\bktBB{#2}}

\newcommand{\varc}[3][]{\mathrm{Var}#1\pth{\hspace{1pt}#2\hspace{1.5pt}|\hspace{1.5pt}#3\hspace{1pt}}}

\let\C=\undefined

\newcommand{\e}{\ensuremath{\mathrm{e}}\xspace}
\newcommand{\C}{\ensuremath{\mathbf{C}}\xspace}
\newcommand{\R}{\ensuremath{\mathbf{R}}\xspace}
\newcommand{\Q}{\ensuremath{\mathbf{Q}}\xspace}
\newcommand{\N}{\ensuremath{\mathbf{N}}\xspace}
\newcommand{\Z}{\ensuremath{\mathbf{Z}}\xspace}

\newcommand{\indi}{\ensuremath{\mathbf{1}}\xspace}
\newcommand{\eps}{\varepsilon}

\newcommand{\vsp}{\vspace{.15cm}}

\newcommand{\hem}{\hspace{1em}}
\newcommand{\hex}{\hspace{1ex}}

\endlocaldefs

\begin{document}

\begin{frontmatter}

\title{Some sample path properties of multifractional Brownian motion}
\runtitle{Some sample path properties of multifractional Brownian motion}

\author{\fnms{Paul} \snm{Balan\c{c}a} %
\ead[label=e1]{paul.balanca@gmail.com}%
}%

\address{\printead{e1}\\[1em]
\'Ecole Centrale Paris\\
Laboratoire MAS, ECP\\
Grande Voie des Vignes\\
92295 Ch\^atenay-Malabry, France
}

\affiliation{\'Ecole Centrale Paris}
\runauthor{Paul Balan\c{c}a}

\begin{abstract}
The geometry of the multifractional Brownian motion (mBm) is known to present a complex and surprising form when the Hurst function is greatly irregular. Nevertheless, most of the literature devoted to the subject considers sufficiently smooth cases which lead to sample paths locally similar to a fractional Brownian motion (fBm). The main goal of this paper is therefore to extend these results to a more general frame and consider any type of continuous Hurst function. More specifically, we mainly focus on obtaining a complete characterization of the pointwise Hölder regularity of the sample paths, and the Box and Hausdorff dimensions of the graph. These results, which are somehow unusual for a Gaussian process, are illustrated by several examples, presenting in this way different aspects of the geometry of the mBm with irregular Hurst functions.
\end{abstract}

\begin{keyword}[class=AMS]
  \kwd{60G07}
  \kwd{60G17}
  \kwd{60G22}
  \kwd{60G44}
\end{keyword}

\begin{keyword}
  \kwd{2-microlocal analysis}
  \kwd{Box dimension}
  \kwd{Hausdorff dimension}
  \kwd{Hölder regularity}
  \kwd{multifractional Brownian motion}
\end{keyword}

\end{frontmatter}


\section{Introduction}

The multifractional Brownian motion (mBm) has been independently introduced by \citet{Peltier.LevyVehel-1995} and \citet{Benassi.Jaffard.ea-1997} as a natural extension of the well-known fractional Brownian motion (fBm). The main idea behind these two works was to drop the stationary assumption on the process, and allow the Hurst exponent to change as time passes. In this way, the mBm is parametrized by a function $t\mapsto H(t)$, usually continuous, and happens to be an interesting stochastic model for non-stationary phenomena (e.g. signal processing, traffic on internet, terrain modelling,~\dots).

Since its introduction, several authors have investigated sample path properties of the mBm. For instance, \citet{Benassi.Jaffard.ea-1997} and \citet{Peltier.LevyVehel-1995} have respectively studied the law of iterated logarithm and the Hölder regularity of its trajectories. In the latter, the Box and Hausdorff dimensions of the graph have also been determined. Moreover, the fine covariance structure have been analysed and refined by \citet{Ayache.Cohen.ea-2000}, \citet{Herbin-2006} and \citet{Stoev.Taqqu-2006}. Finally, the local time and the Hausdorff dimension of the level sets have recently been considered by \citet{Boufoussi.Dozzi.ea-2007}. Several multiparameter extensions, which generalise the classic Lévy fractional Brownian motion and fractional Brownian sheet, have also been introduced and investigated by \citet{Herbin-2006} \citet{Meerschaert.Wu.ea-2008} and \citet{Ayache.Shieh.ea-2011}.\vsp

Fine geometric properties obtained in the aforementioned works usually rely on a key assumption $\Hi_0$ on the Hurst function:
\begin{equation}
   \text{$H$ is a $\beta$-Hölder continuous function such that } \sup_{t\in\R} H(t) < \beta.  \tag{$\mathcal{H}_0$}
\end{equation}
Owing to this hypothesis, the study of the $\Hi_0$-mBm is usually easier and leads to results closely related to their counterparts on fractional Brownian motion. This structure of the mBm is due to the form of its covariance which, under $\Hi_0$, is locally equivalent to fBm's one with Hurst exponent $H(t)$:
\[
  \forall u,v\in B(t,\rho);\quad \espb{(X_u - X_v)^2} \asymp \abs{u-v}^{2H(t)},
\]
for any $t\in\R\setminus\brc{0}$.

The study of the more general multifractional Brownian motion, i.e. when the assumption $\Hi_0$ does not hold, has only been recently considered by \citet{Herbin-2006} and \citet{Ayache-2013}. Both have investigated the Hölder regularity of the sample paths, and more precisely the pointwise and local Hölder exponents, which are defined by
\begin{align}  \label{eq:pointwise_holder}
  \forall t\in\R;\quad\alpha_{X,t}
  &= \sup\brcbb{\alpha : \limsup_{\rho\rightarrow 0} \sup_{u,v\in B(t,\rho)} \frac{\abs{X_u-X_v}}{\rho^\alpha} < \infty }
\end{align}
and
\begin{align}  \label{eq:local_holder}
  \forall t\in\R;\quad\widetilde\alpha_{X,t}
  = \sup\brcbb{\alpha : \limsup_{\rho\rightarrow 0} \sup_{u,v\in B(t,\rho)} \frac{\abs{X_u-X_v}}{\abs{u-v}^\alpha} < \infty }.
\end{align}
These two exponents aim to characterise the local Hölder behaviour of the trajectories at $t$, by respectively comparing its oscillations to a power of $\rho$ or $\abs{u-v}$. For instance, it is well-known that a fractional Brownian motion $B^H$ satisfies with probability one and for all $t\in\R$, $\alpha_{B^H,t} = \widetilde\alpha_{B^H,t} = H$.

The general form of the Hölder regularity of the multifractional Brownian motion has first been obtained by \citet{Herbin-2006}, proving that for all $t > 0$
\begin{equation}
  \alpha_{X,t} = H(t) \wedge \alpha_{H,t} \quad\text{ and }\quad \widetilde\alpha_{X,t} = H(t) \wedge \widetilde\alpha_{H,t} \quad\text{a.s.}
\end{equation}
The previous result only holds for a fixed $t>0$, and thus not uniformly on the sample paths. We also note that the geometric properties of a generic mBm are not a simple extension of those on fBm, but display a more complicated structure where the fine geometry of the Hurst function intervene.
\citet{Ayache-2013} has recently extended this study, obtaining a uniform characterisation of the local Hölder exponent and proving that its pointwise exponent can behave randomly as time passes. The resulting process, with unusual sample path properties, is sometimes called the \emph{irregular multifractional Brownian motion} in the literature.

The main goal of this work is therefore to push further the study of this irregular mBm. In Section~\ref{sec:mbm_representation}, we first discuss the existence of an alternative deterministic representation of the fractional Brownian field, which will then be used to study more precisely the geometric properties of the mBm. Section~\ref{sec:mbm_regularity} is devoted to the 2-microlocal and Hölder regularity of the trajectories (Theorem~\ref{th:2ml_mbm} and Proposition~\ref{prop:pointwise_mbm}), and therefore extends the results obtained by \citet{Herbin-2006} and \citet{Ayache-2013}. Hausdorff and Box dimensions of the graph, and images of fractal sets, are investigated in Section~\ref{sec:mbm_dimension} (Theorems~\ref{th:mbm_dim_box}, \ref{th:mbm_dim_haus} and \ref{th:mbm_dim_images}), where further connections with the geometry of the Hurst function are presented. Appendix~\ref{sec:mbm_appendix} gathers a few deterministic results related to the 2-microlocal frontier and the fractal dimensions which are used along this work. Finally, several examples are also given in these different sections to illustrate the main results and some particular aspects of the geometry of the irregular multifractional Brownian motion.

\section{Deterministic representation}  \label{sec:mbm_representation}

The question of the equivalence of the different mBm representations is known to be non-trivial. On one side, \citet{Peltier.LevyVehel-1995} have introduced the mBm as the following stochastic integral
\[
  \forall t\in\R;\quad X_t = \frac{1}{ \Gamma\pthb{ H(t)+\tfrac{1}{2} } }\int_\R \bktB{ (t-u)_+^{H(t)-1/2} - (-u)_+^{H(t)-1/2} } \dt W_u.
\]
Whereas on the other hand, \citet{Benassi.Jaffard.ea-1997} have considered a real-harmonizable definition:
\[
  \forall t\in\R;\quad \widetilde X_t = \int_\R \frac{ \e^{iut}-1 }{\abs{u}^{H(t)+1/2}} \,\dt\widehat W_u.
\]
\citet{Stoev.Taqqu-2006} have proved that the covariance structure is slightly different between these two definitions. Considering this fact, they have suggested to name multifractional Brownian motion any process $X$ which has the following form
\begin{equation*}
  \forall t\in\R;\quad  X_t = a^+ B^+(t,H(t)) + a^- B^-(t,H(t)),
\end{equation*}
where $(a^+,a^-)\in\R^2\setminus\brc{0,0}$ and the fractional Brownian fields $(t,H)\mapsto B^\pm(t,H)$ are defined by
\begin{equation}  \label{eq:rep_fGf_stoch}
  B^\pm(t,H) = \frac{1}{ \Gamma\pthb{H+\tfrac{1}{2}} }\int_\R \bktB{ (t-u)_\pm^{H-1/2} - (-u)_\pm^{H-1/2} } \dt W_u,
\end{equation}
for all $t\in\R$ and $H\in\ivoo{0,1}$.\vsp

In the next proposition, we present an alternative deterministic representation for these fractional Brownian fields which will be very useful for the study of sample path properties.
\begin{proposition} \label{prop:rep_fGf}
  Suppose $(B_t)_{t\in\R}$ is a Brownian motion such that $B_t \eqdef W(\ivff{0,t})$, where $W$ is a Wiener measure.
  Then, for all $t\in\R$ and any $H\in\ivoo{0,1}$, the fractional Brownian field $B^\pm(t,H)$ is almost surely equal to:
  \begin{itemize}
    \item if $(t,H)\in\R\times\ivoob{\tfrac{1}{2},1}$,
    \begin{align}  \label{eq:rep_fGf_up}
      B^\pm(t,H) = \frac{\pm 1}{ \Gamma\pthb{H - \tfrac{1}{2}} }\int_\R B_u \bktB{ (t-u)_\pm^{H-3/2} - (-u)_\pm^{H-3/2} } \dt u.
    \end{align}

    \item if $(t,H)\in\R\times\ivofb{0,\tfrac{1}{2}}$ and $T$ is a fixed number such that $T\in\ivoo{\pm\infty,t}$,
    \begin{align}  \label{eq:rep_fGf_down1}
      B^\pm(t,H)
      &= \frac{\pm\pthb{ H-\tfrac{1}{2} } }{ \Gamma\pthb{H+\tfrac{1}{2}} }\int_\R (B_u - B_t\indi_{\brc{\pm t\geq \pm T}}) (t-u)_\pm^{H-3/2} - B_u (-u)_\pm^{H-3/2}  \dt u \nonumber \\
      &\pm B_t \pth{t-T}_\pm^{H-1/2} ,
    \end{align}
    or equivalently,
    \begin{align}  \label{eq:rep_fGf_down2}
      B^\pm(t,H)
      &= \frac{\pm 1}{ \Gamma\pthb{H+\tfrac{1}{2}} } \frac{\dt}{\dt s}\pthbb{ \int_{\ivff{\pm T,s}} B_u (s-u)_\pm^{H-\frac{1}{2}} \dt u }\brcb{(t) - (0)} \nonumber \\
      &\pm \frac{{H-\tfrac{1}{2}}}{ \Gamma\pthb{H+\tfrac{1}{2}} } \int_{\ivof{\pm\infty,\pm T}} B_u \bktB{ (t-u)_\pm^{H-3/2} - (-u)_\pm^{H-3/2} } \dt u.
    \end{align}
  \end{itemize}
\end{proposition}
\begin{proof}
  Without any loss of generality, we may only prove the equality for the term $B^+(t,H)$.
  Let first consider the case $t\in\R$ and $H\in\ivoo{\tfrac{1}{2},1}$. Using Ito's lemma on the product $B_u(t-u)_+^{H-1/2}$, we obtain for any $x \leq t$
  \begin{align*}
    B_t (t-t)^{H-1/2} = B_x (t-x)^{H-1/2} + \int_x^t (t-u)^{H-1/2} \dt B_u - (H-\tfrac{1}{2}) \int_x^t B_u (t-u)^{H-3/2} \dt u.
  \end{align*}
  Hence, subtracting the case $t=0$,
  \begin{align*}
    B_x \bktB{ (t-x)_+^{H-1/2} -(-x)_+^{H-1/2} } &= -\int_x^t \bktB{ (t-u)_+^{H-1/2} - (-u)_+^{H-1/2} } \dt B_u \\
    &+ (H-\tfrac{1}{2}) \int_x^t B_u \bktB{ (t-u)_+^{H-3/2} - (-u)_+^{H-3/2} } \dt u.
  \end{align*}
  The Brownian motion is known to satisfy $\lim_{u\rightarrow\pm\infty} \abs{B_u} / \abs{u}^{1/2+\eps} = 0$. Moreover, we observe that $\bktb{ (t-x)_+^{H-1/2} -(-x)_+^{H-1/2} } \sim_{-\infty} (H-\tfrac{1}{2}) t^{H-1/2} (-x)^{H-3/2}$. Therefore, as  $\tfrac{3}{2} - H > \tfrac{1}{2}$,
  \[
    B_x \bktb{ (t-x)_+^{H-1/2} -(-x)_+^{H-1/2} } \xrightarrow[x\rightarrow -\infty]{a.s.} 0.
  \]
  Similarly, using the dominated convergence theorem,
  \[
    \int_x^t B_u \bktB{ (t-u)_+^{H-3/2} - (-u)_+^{H-3/2} } \dt u \xrightarrow[x\rightarrow -\infty]{a.s.} \int_{-\infty}^t B_u \bktB{ (t-u)_+^{H-3/2} - (-u)_+^{H-3/2} } \dt u,
  \]
  Finally, owing to the $L^2$-continuity of the stochastic integral,
  \[
    \int_x^t \bktB{ (t-u)_+^{H-1/2} - (-u)_+^{H-1/2} } \dt B_u \xrightarrow[x\rightarrow -\infty]{L^2(\Omega)} \int_{-\infty}^t \bktB{ (t-u)_+^{H-1/2} - (-u)_+^{H-1/2} } \dt B_u,
  \]
  which proves the expected equality.

  Let now consider the second case $H\in\ivoo{ 0,\tfrac{1}{2} }$. The calculus is quite similar, with only a slight modification in the application of Ito's lemma between $T$ and $t-\eps$.
  \begin{align*}
    &( B_{t-\eps} - B_t ) \eps^{H-1/2} + B_t (t-T)^{H-1/2} \\
    &= B_T (t-T)^{H-1/2} + \int_T^{t-\eps} (t-u)^{H-1/2} \dt B_u - (H-\tfrac{1}{2}) \int_T^{t-\eps} ( B_u-B_t ) (t-u)^{H-3/2} \dt u,
  \end{align*}
  for any $\eps>0$. Owing to the Hölder continuity of the sample paths of $B$, the previous expression converges when $\eps\rightarrow 0$.
  In addition, the calculus presented in the case $H > \tfrac{1}{2}$ holds as well on the interval $\ivff{x,T}$, therefore proving Equation~\eqref{eq:rep_fGf_down1}.
  Finally, the second form \eqref{eq:rep_fGf_down2} of the fractional Brownian field is obtained using an alternative representation of the Riemann--Liouville fractional derivative of a function on the interval $\ivff{T,t}$ (see the book of \citet{Samko.Kilbas.ea-1993} for more information on the subject).
\end{proof}
This type of representation has already been exhibited and studied by \citet{Picard-2011} for the fractional Brownian motion, even though his proof is based on a different approach (approximation of the Brownian motion with Lipschitz functions). The case $H>\tfrac{1}{2}$ has also been considered by \citet{Takashima-1989}.

As presented in the next simple lemma and owing to the Hölder continuity of the sample paths of Brownian motion, we obtain a deterministic representation of the fractional Brownian field $(t,H)\mapsto B(t,H)$ which is continuous in both variables. It will be very useful in the two next sections to study in the sample path properties of the multifractional Brownian motion.
\begin{lemma} \label{lemma:cont_fGf}
  Let $(B_t)_{t\in\R}$ be a continuous Brownian motion and $B(t,H)$ be the fractional Brownian field defined in Proposition~\ref{prop:rep_fGf}.

  Then, with probability one for all $(t_0,H_0)\in\R\times\ivoo{0,1}$, the field $(t,H)\mapsto B(t,H)(\omega)$ is continuous at $(t_0,H_0)$.
\end{lemma}
\begin{proof}
  As previously, we only need to consider the component $B^+(t,H)$.
  The continuity on the domains $\R\times\ivoo{0,\tfrac{1}{2}}$ and $\R\times\ivoo{\tfrac{1}{2},1}$ is a simple consequence of the dominated convergence theorem: the deterministic integral is split in two parts on which the theorem can be applied with different bounds.

  To threat the case $H_0 = \tfrac{1}{2}$, we observe that the representation~\eqref{eq:rep_fGf_down1} in fact holds for all $H\in\ivoo{0,1}$. Then, similarly to the previous case, the dominated convergence theorem implies that the first term converges to zero when $H\rightarrow\tfrac{1}{2}$ whereas the second one simply converges to $B_{t_0}$.
\end{proof}

\subsection{Well-balanced case}

As observed by \citet{Stoev.Taqqu-2006}, the introduction of the mBm is more delicate in the well-balanced case $a^+=a^-$. Indeed, when $H=\tfrac{1}{2}$, the aforementioned definition \eqref{eq:rep_fGf_stoch} of the fractional Brownian field $B(t,H)$ implies that $B(\cdot,\tfrac{1}{2}) = 0$ almost surely. Hence, as presented in \cite{Stoev.Taqqu-2006}, to obtain a well-defined mBm, one needs to change the normalisation term, and more specifically, add a diverging factor: $1 / (H-\tfrac{1}{2})$.

Considering this slightly modified definition of the well-balanced mBm, we observe that it corresponds to a particular deterministic expression.
\begin{proposition}  \label{prop:rep_fGf_balanced}
  Suppose $(B_t)_{t\in\R_+}$ is a continuous Brownian motion and $a^+=a^-=1$. Then, for all $(t,H)\in\R\times\ivoo{0,1}$,
  \begin{equation}  \label{eq:fGf_balanced}
    B(t,H) = \frac{1}{\Gamma\pthb{H+\tfrac{1}{2}}} \int_\R B_u \bktB{ (t-u)^{<H-3/2>} - (-u)^{<H-3/2>} } \dt u\quad\text{a.s.}
  \end{equation}
  where for any $x\in\R$, $x^{<\alpha>} \eqdef \sign(x)\,\abs{x}^\alpha$.
\end{proposition}
\begin{proof}
  Let $(t,H)\in\R\times\ivoo{0,1}$, $\delta > 0$ and $T_\pm = t\mp\delta$.
  Owing to the representation obtained in Proposition~\ref{prop:rep_fGf} and using the well-balanced re-normalisation,
  \[
    B(t,H) = \frac{1}{\Gamma\pthb{H+\tfrac{1}{2}}} \int_\R (B_u-B_t\indi_{\abs{t-u}\leq\delta}) \bktB{ (t-u)^{<H-3/2>} - (-u)^{<H-3/2>} } \dt u.
  \]
  Without any restriction, we may assume that $0\notin B(t,\delta)$. Then, let us observe the component of the integral corresponding to the interval $\ivff{t-\delta,t+\delta}$. It is equal to
  \[
    \int_{B(t,\delta)} (B_u-B_t) (t-u)^{<H-3/2>} \dt u
    =\lim_{\eps\rightarrow 0}\int_{B(t,\delta)\setminus B(t,\eps)} (B_u-B_t) (t-u)^{<H-3/2>} \dt u
  \]
  For any $\eps>0$, we note that $\int_{B(t,\delta)\setminus B(t,\eps)} B_t (t-u)^{<H-3/2>} \dt u = 0$. Furthermore, the other part converges as well to
  \begin{align*}
    \int_{B(t,\delta)\setminus B(t,\eps)} B_u (t-u)^{<H-3/2>} \dt u
    = &\int_{\ivff{t-\delta,t-\eps}} (B_u-B_{2t-u}) (t-u)^{H-3/2} \dt u \\
    \longrightarrow_{\eps\rightarrow0} &\int_{\ivff{t-\delta,t}} (B_u-B_{2t-u}) (t-u)^{H-3/2} \dt u,
  \end{align*}
  since $B$ is Hölder continuous at $t$. Hence, the integral converges when $\eps\rightarrow 0$ and we obtain the representation presented in Equation~\eqref{eq:fGf_balanced}.
\end{proof}
On the contrary to the formulas presented in Proposition~\ref{prop:rep_fGf}, we observe that the representation~\eqref{eq:fGf_balanced} is an improper deterministic integral when $H\leq\tfrac{1}{2}$.

Also note that we may check that when $H=\tfrac{1}{2}$, the corresponding Gaussian process is, up to a constant, a Brownian motion. More precisely, \citet{Stoev.Taqqu-2006} have shown that it has the following stochastic integral representation
\[
  B\pthb{ t,\tfrac{1}{2} } = \frac{1}{\Gamma\pthb{H+\tfrac{1}{2}}} \int_\R \bktB{ \log\,\abs{t-u}-\log\,\abs{u} } \dt W_u.
\]
As proved in the next proposition, the well-balanced case is in fact a classic mBm defined with respect to the latter Brownian motion.
\begin{proposition}  \label{prop:rep_fGf_balanced2}
  Suppose $B$ is a continuous Brownian motion and $B(t,H)$ denotes the well-balanced fractional Brownian field. Define the Brownian motion $\widetilde B$ by
  \[
    \forall t\in\R;\quad \widetilde B_t = \frac{1}{2\Gamma\pthb{H+\tfrac{1}{2}}} \int_\R \bktB{ \log\,\abs{t-u}-\log\,\abs{u} } \dt B_u.
  \]
  Then, for all $t\in\R$ and $H\in\ivoo{0,1}$,
  \[
    B(t,H) = \frac{1}{\pthb{H-\tfrac{1}{2}}\tan\pthb{ (H+\tfrac{1}{2})\tfrac{\pi}{2} }}\pthb{ \widetilde B^+(t,H) - \widetilde B^-(t,H) }\quad \text{a.s.}
  \]
  where $\widetilde B^\pm(t,H)$ designates the fractional Brownian field with respect to $\widetilde B$. Note that the normalising function is continuous at $H=\tfrac{1}{2}$.
\end{proposition}
\begin{proof}
  There exist several ways to prove this result. One might directly obtain the expression of the fractional Brownian field $\widetilde B^\pm(t,H)$ using Propositions~\ref{prop:rep_fGf} and \ref{prop:rep_fGf_balanced}, but the calculus is unfortunately a bit tricky.

  The proof happens to be simpler if we consider the Fourier representation of $B(t,H)$ obtained by \citet{Stoev.Taqqu-2006}. For all $t\in\R$ and $H\in\ivoo{0,1}$,
  \[
    B^\pm(t,H) = \frac{1}{2\sqrt{2\pi}} \int_\R \frac{\e^{it\xi}-1}{\abs{\xi}^{H+1/2}} \e^{\mp i\sign(\xi)(H+1/2)\pi/2} \dt \widehat W(\xi),
  \]
  where $\widehat W$ is a complex Gaussian measure such that $\widehat W$ is a complex Wiener measure on $\R_+$ and $\overline{\widehat W(A)} = \widehat W(-A)$ for any $A\in\Bi(R_+)$.
  The well-balanced fractional Brownian field then has the following representation
  \[
    B(t,H) = \frac{2\cos\pthb{ (H+\frac{1}{2})\frac{\pi}{2} }}{\sqrt{2\pi}\pthb{ H-\frac{1}{2} }} \int_\R \frac{\e^{it\xi}-1}{\abs{\xi}^{H+1/2}} \dt \widehat W(\xi).
  \]
  We introduce a slightly modified complex Gaussian measure $\widetilde W$, given by
  \[
    \forall A\in\Bi(\R_+);\quad \widetilde W(A) \eqdef i \widehat W(A)\quad\text{and}\quad \widetilde W(-A) \eqdef \overline{\widetilde W(A)} = -i \widehat W(-A).
  \]
  Then, let $\widetilde B^\pm(t,H)$ denote the fractional Brownian field defined with respect to $\widetilde W$. Since $\widetilde W(\dt \xi) = i\sign(\xi)\widehat W(\dt \xi)$,
  \begin{align*}
    \widetilde B^+(t,H) - \widetilde B^-(t,H)
    &= \frac{ 2\sin\pthb{ (H+\frac{1}{2})\frac{\pi}{2} } }{\sqrt{2\pi}} \int_\R \frac{\e^{it\xi}-1}{\abs{\xi}^{H+1/2}} (-i) \sign(\xi) \,\dt \widetilde W(\xi) \\
    &= \pthb{H-\tfrac{1}{2}}\tan\pthb{ (H+\tfrac{1}{2})\tfrac{\pi}{2} } \, B(t,H).
  \end{align*}
  To conclude the proof, we need to check that the field $\widetilde B^\pm(t,H)$ corresponds to the same process defined in the proposition, i.e. the integral with respect to the Brownian motion $\widetilde B$.
  For this purpose, we simply need to consider the previous equality when $H=\tfrac{1}{2}$. We first note that $\lim_{H\rightarrow 1/2} \pthb{H-\tfrac{1}{2}}\tan\pthb{ (H+\tfrac{1}{2})\tfrac{\pi}{2} } = 1$. Hence,
  \[
    \widetilde B_t \eqdef \frac{1}{2} B\pthb{t,\tfrac{1}{2}} = \frac{1}{2} \pthB{ \widetilde B_+\pthb{t,\tfrac{1}{2}} - \widetilde B_-\pthb{t,\tfrac{1}{2}} }.
  \]
  According to Equation~\eqref{eq:rep_fGf_stoch}, this last equality proves that $\widetilde B$ is the Brownian motion with respect to which are defined the fractional Brownian fields, and thus we clearly obtain the expected representation of $B(t,H)$.
\end{proof}
The equality in law between the two processes can be obtained more directly using Theorem 4.1 from \cite{Stoev.Taqqu-2006}. Nevertheless, Proposition~\ref{prop:rep_fGf_balanced2} indicates precisely how these two fractional Brownian fields are related, and proves that the well-balanced one is simply a classic field with a change of Brownian motion.
Besides, this statement allows us to simply the study of the sample path properties of the mBm since we do not need to consider separately the specific well-balanced case $a^+=a^-$.

\section{Hölder and 2-microlocal regularity}  \label{sec:mbm_regularity}

Following the convention initiated by \cite{Stoev.Taqqu-2006}, we will call \emph{multifractional Brownian motion} a process $X$ which has the following form:
\[
  \forall t\in\R;\quad X_t = a^+ B^+\pthb{ t,H(t) } + a^- B^-\pthb{ t,H(t) },
\]
where $(a^-,a^+)\in\R^2\setminus\brc{(0,0)}$. As previously said, we may also suppose without any loss of generality that $a^+\neq a^-$. The Hurst function $t\mapsto H(t)\in\ivoo{0,1}$ will be assume to be at least a deterministic \cadlag function.

We may nevertheless note that several of the following results remain valid if $H(\cdot)$ is replaced by a stochastic process with \cadlag sample paths. This last model has already been considered by \citet{Papanicolaou.Solna-2003} and \citet{Ayache.Taqqu-2005}, where the latter has obtained some general sample path properties.
For any \cadlag function $H(\cdot)$, the set $\brcb{ t\in\R : \Delta H(t) \neq 0 }$ is countable, and therefore we can easily see from the definition of the mBm that with probability one,
\[
  \forall t\in\R;\quad \Delta H(t)\neq 0 \hem\Longleftrightarrow\hem \Delta X_t\neq 0.
\]
Furthermore, $\Delta X_t$ is a centered Gaussian variable whose variance is proportional to $\Delta H(t)^2$.\vsp

In this section, we will use specific functional spaces, called 2-microlocal spaces, to characterize the fine sample path properties of the mbm. The use of the unusual framework is particularly motivated by the deterministic fractional integrals and derivatives of Brownian motion exhibited in Proposition~\ref{prop:rep_fGf}. Indeed, the latter spaces behave particularly well under the action of fractional operators, which will allow us to deduce many properties on the mBm from the geometry of the Brownian motion.

Historically, the 2-microlocal formalism has been introduced by \citet{Bony-1986} to study the singularities of the solutions of PDEs, and has only been recently introduced in a stochastic frame by \citet{Herbin.LevyVehel-2009} to characterise the fine regularity of processes. Several characterizations of the 2-microlocal spaces have been investigated in the literature, including the following time-domain one due to \citet{Kolwankar.LevyVehel-2002} and \citet{Seuret.LevyVehel-2003}.
\begin{definition}  \label{def:2ml_spaces_01}
  Let $t\in\R$, $s'\leq 0$ and $\sigma\in\ivoo{0,1}$ such that $\sigma-s'\notin\N$. A function $f:\R\rightarrow\R$ belongs to the \emph{2-microlocal space} $C^{\sigma,s'}_t$ if there exist $C>0$, $\rho>0$ and a polynomial $P_t$ such that
  \begin{equation} \label{eq:def_2ml_spaces_01}
    \absb{\pthb{ f(u)-P_t(u) } - \pthb{ (f(v)-P_t(v) } } \leq C\abs{u-v}^\sigma \pthb{ \abs{u-t}+\abs{v-t} }^{-s'},
  \end{equation}
  for all $u,v\in B(t,\rho)$.
\end{definition}
In this work, it will only be necessary to consider the case $\sigma\in\ivoo{0,1}$. Nevertheless, Definition~\ref{eq:def_2ml_spaces_01} can be extended $\sigma\in\R\setminus\Z$. We refer to \cite{Seuret.LevyVehel-2003, Balanca-2013} for a more complete presentation on the subject.

As previously outlined, a very useful properties of the 2-microlocal spaces is their stability under the application of fractional integrations and derivations:
\begin{equation} \label{eq:2ml_spaces_fracint}
  \forall \alpha>0;\quad f\in C^{\sigma,s'}_t \quad\Longleftrightarrow\quad  I_\pm^\alpha f \in C^{\sigma+\alpha,s'}_t,
\end{equation}
where $I_\pm^\alpha f$ denotes the fractional integral of $f$ of order $\alpha$: $\pthb{I_\pm^\alpha f}(t) \eqdef \int_\R (t-u)_\pm^{\alpha-1} f(u) \,\dt u$.

Following the presentation of these functional spaces, we can introduce a related regularity tool, named the \emph{2-microlocal frontier} and defined by
\begin{equation}
  \forall s'\in\R;\quad \sigma_{f,t}(s') = \sup\brcb{\sigma\in\R : f\in C^{\sigma,s'}_{t}}.
\end{equation}
We observe that when using the 2-microlocal formalism, the local regularity at $t$ is characterised by a function, not a single coefficient. Owing to inclusion properties of the 2-microlocal spaces, the map $s'\mapsto\sigma_{f,t}(s')$ is well-defined and satisfies several interesting properties:
\begin{itemize}  \itemsep1pt
  \item $\sigma_{f,t}(\cdot)$ is a concave non-decreasing function;
  \item $\sigma_{f,t}(\cdot)$ has left and right derivatives between $0$ and $1$.
\end{itemize}
Furthermore, if the modulus of continuity of $f$ is such that $\omega_f(h) = \mathrm{O}\,(1/\abs{\log(h)})$, the pointwise and local Hölder exponents can be retrieved from the frontier:
\[
  \alpha_{f,t} = -\inf\brc{s' : \sigma_{f,t}(s')\geq 0} \quad\text{and}\quad \widetilde\alpha_{f,t} = \sigma_{f,t}(0).
\]
Finally, owing to the stability of the 2-microlocal spaces, the application of a fractional integration simply corresponds to a translation of the 2-microlocal frontier along the vertical axis:
\begin{equation*}
  \forall s'\in\R;\quad \sigma_{I_\pm^\alpha f,t}(s') = \sigma_{f,t}(s')+\alpha.
\end{equation*}
The 2-microlocal frontier of some Gaussian processes has been determined by \citet{Herbin.LevyVehel-2009}. For instance, a fractional Brownian motion $B^H$, $H\in\ivoo{0,1}$ satisfies with probability one and for all $t\in\R$:
\[
  \forall s'\in\R;\quad \sigma_{B^H,t}(s') = \pthb{H + s'}\wedge H.
\]

To investigate the 2-microlocal and Hölder regularity of the multifractional Brownian motion, we first need to properly study the behaviour of the fractional Brownian field and its successive partial derivatives. Hence, in the next proposition, we consider uniformly on the variables $t$ and $H$ the 2-microlocal frontier of $\partial^k_H B(t,H)$ along the time axis.
\begin{proposition}  \label{prop:2ml_fGf}
  Suppose $B(t,H)$ is a fractional Brownian field. For any $H\in\ivoo{0,1}$ and any $k\in\N$, let $B^{H,k}$ denotes the process $t\mapsto \partial^k_H\, B(t,H)$.
  Then, with probability one, for all $H\in\ivoo{0,1}$, $k\in\N$ and $t\in\R$
  \begin{equation}
    \forall s'\in\R;\quad \sigma_{B^{H,k},t}(s') = \pthb{H + s'}\wedge H.
  \end{equation}
\end{proposition}
\begin{proof}
  Since we use the deterministic representation obtained in Proposition~\ref{prop:rep_fGf}, we can set a fix $\omega\in\Omega$ in the proof.
  For the sake of readability, let us first consider $H > \tfrac{1}{2}$ and remove the normalisation constant $1/\Gamma\pthb{ H-\tfrac{1}{2} }$. Then, for any $k\in\N$, the partial derivative with respect to $H$ is equal to
  \[
    \partial^k_H\, B^\pm (t,H) = \int_\R B_u \bktB{ (t-u)_\pm^{H-3/2} \log^k(t-u)_\pm - (-u)_\pm^{H-3/2} \log^k (-u)_\pm } \dt u.
  \]
  Suppose $\ivoo{T_+,T_-}$ is a non-empty open interval. Since the 2-microlocal frontier is a purely local property of the sample paths, we may restrict our study to $t\in\ivoo{T_+,T_-}$ with any loss of generality. Then, we note that the component
  \[
    t\mapsto \int_{\ivof{\mp\infty,T_\pm}} B_u \bktB{ (t-u)_\pm^{H-3/2} \log^k(t-u)_\pm - (-u)_\pm^{H-3/2} \log^k (-u)_\pm } \dt u
  \]
  is a smooth function, and therefore does not have any influence on the 2-microlocal regularity. Hence, we focus on the study of the remaining integral
  \[
    t\in\ivoo{T_+,T_-}\mapsto \int_{\ivff{T_+,T_-}} B_u (t-u)_\pm^{H-3/2} \log^k(t-u)_\pm \dt u,
  \]
  when $H>\tfrac{1}{2}$. Similarly, the case $H<\tfrac{1}{2}$ is reduced to the analysis of
  \[
    t\in\ivoo{T_+,T_-}\mapsto \frac{\dt}{\dt s}\pthbb{ \int_{\ivff{T_\pm,s}} B_u (s-u)_\pm^{H-\frac{1}{2}} \log^k(t-u)_\pm \dt u }(t).
  \]
  If we define the process $Y_t = B_t\indi_{t\in\ivff{T_+,T_-}}$, we observe that the two previous formulas correspond respectively to fractional-like integral and derivative of $Y$ on the real axis. Compared to the usual definition, there is only a slight modification in the kernel with a logarithmic term.

  Furthermore, when $H<\tfrac{1}{2}$, the fractional-like derivative is defined as the classic derivative of a  integral of order $H+\tfrac{1}{2}$. Since we know the behaviour of the 2-microlocal regularity under derivation, we may restrict our study to the fine analysis of the following process
  \[
    t\in\ivoo{T_+,T_-}\longmapsto \pthb{ I^{\alpha,k}_\pm Y}(t) \eqdef \int_{\R} Y_u \,(t-u)_\pm^{\alpha-1} \log^k(t-u)_\pm \dt u,
  \]
  for any $\alpha\in\ivoo{0,1}$ and $k\in\N$.

  We expect the 2-microlocal frontier of this integral to behave similarly to the one of a classic fractional integral of $Y$, since the only difference lies in a logarithmic term. This easiest way to prove this statement is to compute the Fourier transform of the corresponding pseudo-differential operator and show it has an effect equivalent to the fractional integration in the Fourier space.

  The computation of Fourier transform of the tempered distribution $u^{\alpha-1}_\pm \log^k u_\pm$ is similar to the classic calculus on $u^{\alpha-1}_\pm$. We first consider the Laplace transform, with $\Re(z) > 0$. Since this transform is analytic on the previous domain, we can restrict ourselves to the case that $z > 0$. Then,
  \begin{align*}
    \forall z>0;\quad \int_{\R_+} u^{\alpha-1} \log^k(u) \,\e^{-z u} \,\dt u
    &= z^{-\alpha} \int_{\R_+} v^{\alpha-1}\pthb{ \log v - \log z }^k \e^{-v} \dt v \\
    &= z^{-\alpha} \sum_{j=0}^k (-1)^j \log^j(z) \,C^j_k \underbrace{ \int_{\R_+} v^{\alpha-1} \log^{k-j}(v) \,\e^{-v} }_{F(\alpha,k-j)}.
  \end{align*}
  As the last term is also analytic, it implies the equality for any $z\in\C$ such that $\Re(z) > 0$.

  To obtain the Fourier transform of $u^{\alpha-1}_\pm \log^k u_\pm$ as a tempered distribution, we simply have to consider $z = i\xi + \eps$ and let $\eps\rightarrow 0$. The scalar product with a test function converges, and proves that the generalized Fourier transform is equal to
  \begin{align*}
    &\Fi_g\pthb{u^{\alpha-1}_+ \log^k u_+}(\xi) \\
    &= e^{-i\sign(\xi)\tfrac{\pi}{2}\alpha} \abs{\xi}^{-\alpha} \frac{1}{\sqrt{2\pi}} \sum_{j=0}^k (-1)^j  \pthb{\log\abs{\xi}+i\sign(\xi)\tfrac{\pi}{2} }^j  \cdot C^j_k F(\alpha,k-j).
  \end{align*}
  Therefore, the Fourier transform corresponds to a weighted sum of elements $\abs{\xi}^{-\alpha}\log^n\abs{\xi}$, with a phase correction term.

  To express the action of these pseudo-differential operators on the regularity of the Brownian motion, we rely on the original definition of the 2-microlocal spaces in the Fourier space and make use of the arguments presented by \citet{Jaffard-1991} to prove the stability.
  More precisely, the aforementioned definition is based on the Littlewood--Paley decomposition of tempered distributions. Shortly, consider $\phi\in\Si(\R)$ such that its Fourier transform satisfies
  \begin{align*}
    \widehat\phi(\xi) =
    \begin{cases}
      \, 1 &\text{si}\hem \abs{\xi}\leq 1/2; \\[0.2ex]
      \, 0 &\text{si}\hem \abs{\xi}\geq 1,
    \end{cases}
  \end{align*}
  and define for any tempered distribution $f\in\Si'(\R)$ and for all $j\in\N$, $\phi_j(x) = 2^{j}\phi\pth{2^j x}$, $\psi_j = \phi_{j+1} - \phi_j$, $S_0 f = \phi \ast f$ and $\Delta_j f = \psi_j \ast f$. Then, the Littlewood--Paley decomposition tells us that $f$ can be expressed as following
  \[
    f = S_0 f + \sum_{j=0}^{\infty} \Delta_j f,
  \]
  where we observe that every term in the sum belongs to $L^2(\R)$. Roughly speaking, this decomposition corresponds to the application of band-pass filters in the Fourier space. As presented by \citet{Bony-1986,Meyer-1998}, $f\in C^{\sigma,s'}_{t}$ if and only there exists $C>0$ such that for all $j>0$,
  \[
    \abs{\Delta_j f(u)} \leq  C\,2^{-j\sigma} \pthb{2^{-j} + \abs{u - t}}^{-s'}.
  \]
  Observe that, up to a simple translation, we may assume that $t=0$. Then, define the rescaled function $U_j(x) = (\Delta_j f)(2^{-j}u)$. Its Fourier transform $\widehat U_j$ is carried by $\brc{\xi : \tfrac{1}{2}<\abs{\xi}<2}$ and $U_j$ is such that $\abs{U_j(u)}\leq C\,2^{-j(\sigma-s')} \pthb{1 + \abs{u}}^{-s'}$.
  Applying the operator $\abs{\xi}^{-\alpha}\log^k\abs{\xi}$ in the Fourier space on $f$, we obtain a corresponding collection of rescaled functions $(V_j)_{j>0}$ which satisfy for every $j>0$,
  \begin{align*}
    \widehat V_j(\xi) = \widehat U_j(\xi) \cdot \bktbb{ \sum_{\ell=0}^k c(\ell,k) \, \abs{\xi}^{-\alpha}\log^\ell\abs{\xi}  \cdot  j^{k-\ell}\, 2^{-j\alpha} },
  \end{align*}
  where $c(\ell,k)$ is a constant only depending on $\ell$ and $k$. Note that similarly to the classic fractional integral, the Fourier multipliers $\widehat K_{\alpha,\ell}(\xi) \eqdef \abs{\xi}^{-\alpha} \log^\ell\abs{\xi}$ coincide with Schwartz functions on $\tfrac{1}{2}<\abs{\xi}<2$. Coming back into the time space, $V_j$ is given by
  \[
    V_j = \sum_{\ell=0}^k j^{k-\ell} 2^{-j\alpha} \, c(\ell,k) K_{\alpha,\ell} \ast U_j,
  \]
  where $K_{\alpha,\ell}\in\Si(\R)$ for every $\ell$.
  The convolution with a Schwartz function $K_{\alpha,\ell}$ preserves the polynomial decay, hence for every $j>0$, $\abs{V_j(u)}\leq C\,j^k\,2^{-j(\sigma+\alpha-s')} \pthb{1 + \abs{u}}^{-s'}$. Rescaling by a factor $2^j$, the latter upper bound leads to the characterization of 2-microlocal spaces, proving that
  \[
    \forall s',\sigma\in\R;\quad Y\in C^{\sigma,s'}_t \Longrightarrow I^{\alpha,k}_\pm Y\in C^{\sigma-\eps,s'}_t
  \]
  for any $\eps>0$. Note that the property is slightly weaker than the classic stability under the action of fractional integration due to the logarithmic terms.

  The converse case is obtained similarly,
  \[
    \forall s',\sigma\in\R;\quad I^{\alpha,k}_\pm Y\in C^{\sigma,s'}_t \Longrightarrow Y\in C^{\sigma-\eps,s'}_t.
  \]
  The same property holds for any weighted sum of $I^{\alpha,k}_- Y$ and $I^{\alpha,k}_+ Y$, and therefore, the 2-microlocal frontier of $I^{\alpha,k}_\pm Y$ satisfies for any $H\in\ivoo{0,1}$ and $k\in\N$
  \[
    \forall s'\in\R;\quad \sigma_{I^{\alpha,k}_\pm Y,t}(s') = \sigma_{Y,t}(s') + H-\frac{1}{2} = (s'+H)\wedge H.
  \]
  It concludes the proof, since we have shown previously that the 2-microlocal frontier of $B^{H,k}$ is equal to the regularity $I^{\alpha,k}_\pm Y$.
\end{proof}
\citet{Ayache-2013} has obtained a similar result on the field $(t,H)\mapsto B^+(t,H)$ based on a multiresolution analysis of the latter.

As outlined in the introduction, the Hölder regularity of a general multifractional Brownian motion has already been investigated by \citet{Herbin-2006} and \citet{Ayache-2013}. In particular, the latter has proved that with probability one, the local Hölder exponent is given by
\[
  \forall t\in\R\setminus\brc{0};\quad \widetilde\alpha_{X,t} = H(t)\wedge\widetilde\alpha_{H,t}.
\]
Note that the behaviour at $t=0$ is slightly different as the form of the variance of increments differs at this point.\vsp

Let us first obtain a estimate for the general 2-microlocal frontier. For this purpose, we need to introduce the multiplicity of the component $h\mapsto \pthb{ B(t,h)-B(t,H) }$ of the fractional field at $(t,H)\in\R\times\ivoo{0,1}$, i.e.
\[
  m_{t,H} = \inf\brcb{ k\in\N\setminus\brc{0} : \partial_H^{k} B(t,H) \neq 0 }.
\]
\begin{theorem}  \label{th:2ml_mbm}
  Suppose $X$ is a multifractional Brownian motion with Hurst function $H$.
  Then, with probability one, the 2-microlocal frontier of $X$ at any $t\in\R$ satisfies
  \begin{align*}
    \sigma_{X,t}(s') \geq
    \pthb{ s' + H(t) }\wedge H(t)\wedge
    \sigma_{H,t}(s' + H(t))\wedge m_{t,H(t)} \sigma_{H,t}(s' / m_{t,H(t)}).
  \end{align*}
  for all $s'\geq -\alpha_{X,t}$. Furthermore, when $m_{t,H(t)} = 1$, the frontier is equal to
  \begin{align*}
    \forall s'\geq -\alpha_{X,t};\quad \sigma_{X,t}(s') = \pthb{ s' + H(t) }\wedge H(t)\wedge\sigma_{H,t}(s').
  \end{align*}
\end{theorem}
\begin{proof}
  Without any loss of generality, we can assume that $\widetilde\alpha_{H,t} > 0$, since otherwise the lower bound is equal to $0$.
  Then, let $\eps>0$, $\rho>0$, $t\in\R$ and $u,v$ be in the neighbourhood $B(t,\rho)$. We first estimate the size of the increments $X_u - X_v$:
  \begin{align*}
    B(u,H_u) - B(v,H_v)
    &= \pthb{ B(v,H_u) - B(u,H_u) }
    + \pthb{ B(v,H_v) - B(v,H_u) }.
  \end{align*}
  Using a simple integration by parts, we observe that the first term is equal to
  \begin{align*}
    B(v,H_u) - B(u,H_u)
    &= B(v,H_t) - B(u,H_t) + \int_{H_t}^{H_u} \pthb{ \partial_H B(v,h) - \partial_H B(u,h) } \dt h.
  \end{align*}
  Owing to Lemma \ref{prop:2ml_fGf}, for any $\eps>0$, there exists $c_1 > 0$ such that
  \[
    \forall u,v\in B(t,\rho),\ \forall h\in\ivff{H_t,H_u};\quad \absb{ \partial_H B(v,h) - \partial_H B(u,h) } \leq c_1 \abs{u-v}^{h-\eps}.
  \]
  Therefore, the integral term satisfies
  \[
    \absbb{ \int_{H_t}^{H_u} \pthb{ \partial_H B(v,h) - \partial_H B(u,h) } \dt h } \leq c_1 \abs{ H_t - H_u } \cdot \abs{v-u}^{H_t\wedge H_u-\eps},
  \]
  and is thus negligible in front of the first one. Still using Lemma \ref{prop:2ml_fGf}, we know that for all $u,v\in B(t,\rho)$
  \[
    \abs{ B(v,H_t) - B(u,H_t) } \leq c_2 \abs{u-v}^\sigma\pthb{\abs{u-t}+\abs{v-t}}^{-s'},
  \]
  where $s'\geq -H_t$ and $\sigma < (H_t+s')\wedge H_t$. This inequality corresponds to the first term in the lower bound.

  Using Taylor's formula, we can obtain an estimate of the second component $\pth{ B(v,H_v) - B(v,H_u) }$.
  \begin{align*}
    B(v,H_v) - B(v,H_u) &= \sum_{k=1}^{m-1} \frac{ \partial_H^{k} B(v,H_u) }{k!}(H_v  - H_u)^k
    + \int_{H_u}^{H_v} \frac{ \partial_H^{m} B(v,h) }{(m-1)!}(H_v-h)^{m-1} \dt h,
  \end{align*}
  where $m$ denotes the coefficient $m_{t,H_t}$ previously introduced.
  There exists a neighbourhood $B(t,\rho)$ of $t$ such that for every $k\in\brc{1,\dotsc,m-1}$ and for all $v\in B(t,\rho)$, $\abs{ \partial_H^{k} B(v,H_u) } \leq c_3 \abs{v-t}^{H_t-\eps}$. Hence the term $k=1$ is dominant in the sum and for all $u,v\in B(t,\rho)$
  \[
    \abs{ \partial_H B(v,H_u) }\cdot \abs{ H_v - H_u } \leq c_4 \abs{u-v}^\sigma\pthb{\abs{u-t}+\abs{v-t}}^{-s+H_t-\eps},
  \]
  for any $s'\geq -\alpha_{H,t}$ and $\sigma < \sigma_{H,t}(s')$. This inequality leads to the second term in the lower bound.

  Finally, let consider the last integral. $\partial_H^{m} B(v,h)$ can be simply bound by a constant, implying that the integral is upper bounded by $\abs{H_u-H_v}^m$. The third lower bound is a direct consequence of this estimate.\vsp

  To end the proof, we study the case $m_{t,H(t)} = 1$. The lower bound is clearly equal to $(H_t+s')\wedge H_t\wedge \sigma_{H,t}(s')$. Let set $s'\geq -\alpha_{X,t}$. Suppose first that $(H_t+s')\wedge H_t < \sigma_{H,t}(s')$. Then, the term $B(v,H_u) - B(u,H_u)$ will dominate in the increment $X_u-X_v$, therefore implying  the inequality $\sigma_{X,t}(s') \leq (H_t+s')\wedge H_t$. In the other case, we observe that
  \begin{align*}
    B(v,H_v) - B(v,H_u) &= \partial_H B(v,H_v)(H_u - H_v)
    + \int_{H_v}^{H_u} \partial_H^{2} B(v,h)(H_v - h) \dt h,
  \end{align*}
  There exists $\rho>0$ and $c_5 >0$ such that for all $v\in B(t,\rho)$, $\abs{\partial_H B(v,H_v)} > c_5$, implying that the regularity of this component is $\sigma_{H,t}(s')$ Furthermore, as previously noted, the second term is upper-bounded by $(H_u-H_v)^2$, and thus is negligible in front of the first one. Therefore, the increment $X_u-X_v$ behave similarly to $H_u-H_v$, proving that $\sigma_{X,t}(s') \leq \sigma_{H,t}(s')$.
\end{proof}

As observed in the previous proof, the study of the general form of the 2-microlocal frontier is quite technical, in particular because of the interaction which may appear between the oscillations of the Hurst function $H$ and the derivatives of the fractional Brownian field $B(t,H)$.
Nevertheless, a complete estimate can be obtained for the pointwise Hölder regularity.
\begin{proposition}  \label{prop:pointwise_mbm}
  Suppose $X$ is a multifractional Brownian motion with Hurst function $H$.
  Then, with probability one, the pointwise exponent is equal to
  \[
    \forall t\in\R;\quad \alpha_{X,t} = H(t)\wedge m_{t,H(t)} \alpha_{H,t}.
  \]
\end{proposition}
\begin{proof}
  Let us proceed similarly to the proof of Theorem~\ref{th:2ml_mbm} and study the increment $X_u-X_t = \pthb{B(u,H_u)-B(u,H_t)} + \pthb{B(u,H_t) - B(t,H_t)}$. Owing to Lemma~\ref{prop:2ml_fGf}, the pointwise regularity of the second term is exactly $H(t)$.

  Then, we consider a slightly different Taylor expansion on the first increment,
  \begin{align*}
    B(u,H_u) - B(u,H_t) &= \sum_{k=1}^{m} \frac{ \partial_H^{k} B(u,H_t) }{k!}(H_u - H_t)^k
    + \int_{H_t}^{H_u} \frac{ \partial_H^{m+1} B(u,h) }{m!}(H_u-h)^{m} \dt h.
  \end{align*}
  For every $k\in\brc{1,\dotsc,m-1}$, there exists a simple upper bound : $\abs{\partial_H^{k} B(u,H_t)}\abs{H_u - H_t}^k \leq c_1 \abs{u-t}^{H_t+k\alpha_{H,t}-\eps}$. Hence, the pointwise regularity of this type of term is strictly greater than $H(t)$, they are therefore negligible. In the case $k=m$, we note that there exists $\rho>0$ and $c_2>0$ such that for all $u\in B(t,\rho)$, $\abs{\partial_H^{k} B(u,H_t)} \geq c_2$. Hence, the pointwise exponent of this component is exactly $m\cdot\alpha_{X,t}$. Finally, similarly to the previous proof, the last term can be simply upper bounded by $c_3 \abs{H_u-H_v}^{m+1} \leq c_3 \abs{u-t}^{(m+1)\alpha_{H,t}}$. This last estimate is sufficient to conclude the proof.
\end{proof}

The previous result constitutes an extension of the estimate obtained by \citet{Ayache-2013}. As pointed out by the latter, it shows that the multifractional Brownian motion can have a non-deterministic pointwise regularity. More precisely, \citet{Ayache-2013} has proved that for any $\eps>0$, with positive probability,
\[
  \dimH\brc{t\in\ivff{a,b} : m_{t,H(t)} > 1 } \geq 1 - \inf_{u\in\ivff{a,b}} H(u) - \eps.
\]
In the view of the result obtained in Proposition~\ref{prop:pointwise_mbm}, it would be interesting to study the Hausdorff dimension of the following sets
\[
  \forall k\in\N;\quad E_k = \brcb{t\in\ivff{a,b} : m_{t,H(t)} = k}.
\]
This question seems to be much more difficult than the previous estimate, since it involves the successive derivatives $\partial^k_H B(t,H)$ of the fractional field and thus the correlation existing between them.
Finally, we also note that in terms of regularity, the behaviour at $t=0$ is specific, since the multiplicity $m_{0,H(0)}$ is equal to infinity.

To end this section, we present a couple of examples of Hurst functions which illustrate this particular geometry of the multifractional Brownian motion.
\begin{example}  \label{ex:2ml_mbm1}
  Suppose $H:\R\rightarrow\ivoo{\tfrac{1}{2},\tfrac{3}{4}}$ is an Hurst function such that for all $t\in\R$, the 2-microlocal local frontier of $H$ at $t$ is equal to
  \[
    \forall s'\in\R;\quad \sigma_{H,t}(s') = \pthB{ s' + \frac{3}{8} }\wedge \frac{3}{8}.
  \]
  This kind of function can be easily constructed from the sample paths of a fractional Brownian motion with parameter $H=\tfrac{3}{8}$.

  Then, the 2-microlocal frontier of the corresponding multifractional Brownian motion satisfies with probability one:
  \begin{itemize}
    \item if $m_{t,H_t} = 1$,
    \[
      \forall s'\in\R; \quad \sigma_{X,t}(s') = \sigma_{H,t}(s') = \pthB{ s' + \frac{3}{8} }\wedge \frac{3}{8};
    \]
    \item if $m_{t,H_t} > 1$,
    \[
      \forall s'\in\R; \quad \sigma_{X,t}(s') = \pthb{ s' + H(t) }\wedge \frac{3}{8}.
    \]
  \end{itemize}
  These two cases are illustrated by Figure~\ref{sfig:2ml_mbm_ex1}.

  To obtain the complete determination of 2-microlocal frontier, we use Theorem~\ref{th:2ml_mbm}, observing that it implies the inequality
  \begin{align*}
    \sigma_{X,t}(s') &
    \geq \pthb{ s'+H(t) }\wedge H(t) \pthB{ s' + \frac{3}{8} + H(t) }\wedge \frac{3}{8} \wedge \pthB{ s' + \frac{3m_{t,H(t)}}{8} }\wedge \frac{3m_{t,H(t)}}{8} \\
    &= \pthb{ s' + H(t) }\wedge \frac{3}{8},
  \end{align*}
  since $H(t)\in\ivoo{\tfrac{1}{2},\tfrac{3}{4}}$. Then, the upper bound is a consequence of the pointwise exponent obtained in Proposition~\ref{prop:pointwise_mbm} and the following property of the 2-microlocal frontier:
  \[
    \forall s'\in\R;\quad \sigma_{X,t}(s') \leq \liminf_{s\rightarrow t} \widetilde\alpha_{X,t},
  \]
  proved in the Appendix (Lemma~\ref{lemma:2ml_holder_bound}).
\end{example}

\begin{figure}[ht!]
  \centering
  \begin{subfigure}[b]{0.49\textwidth}
    \centering
    \includegraphics[width=\textwidth]{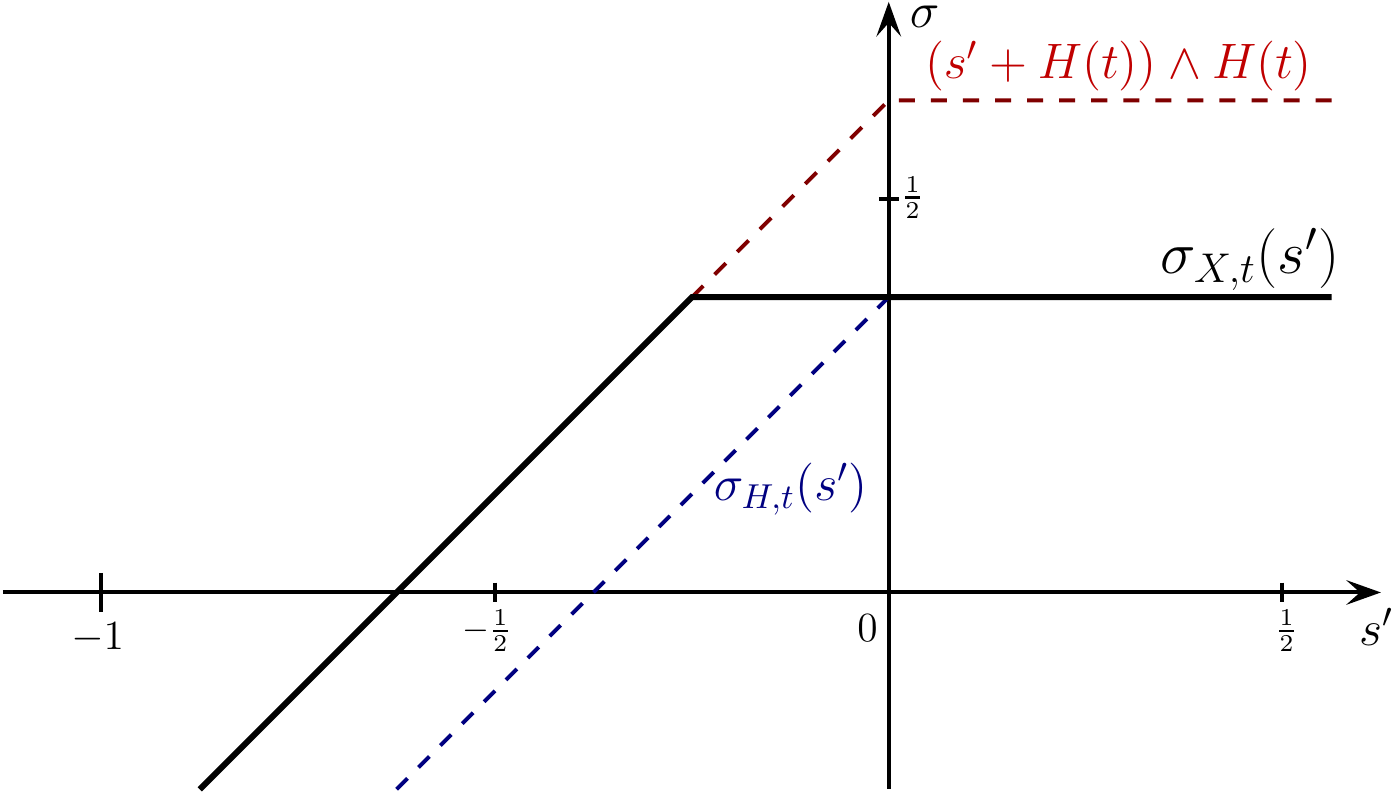}
    \caption{Example 1}
    \label{sfig:2ml_mbm_ex1}
  \end{subfigure}%
  \begin{subfigure}[b]{0.49\textwidth}
    \centering
    \includegraphics[width=\textwidth]{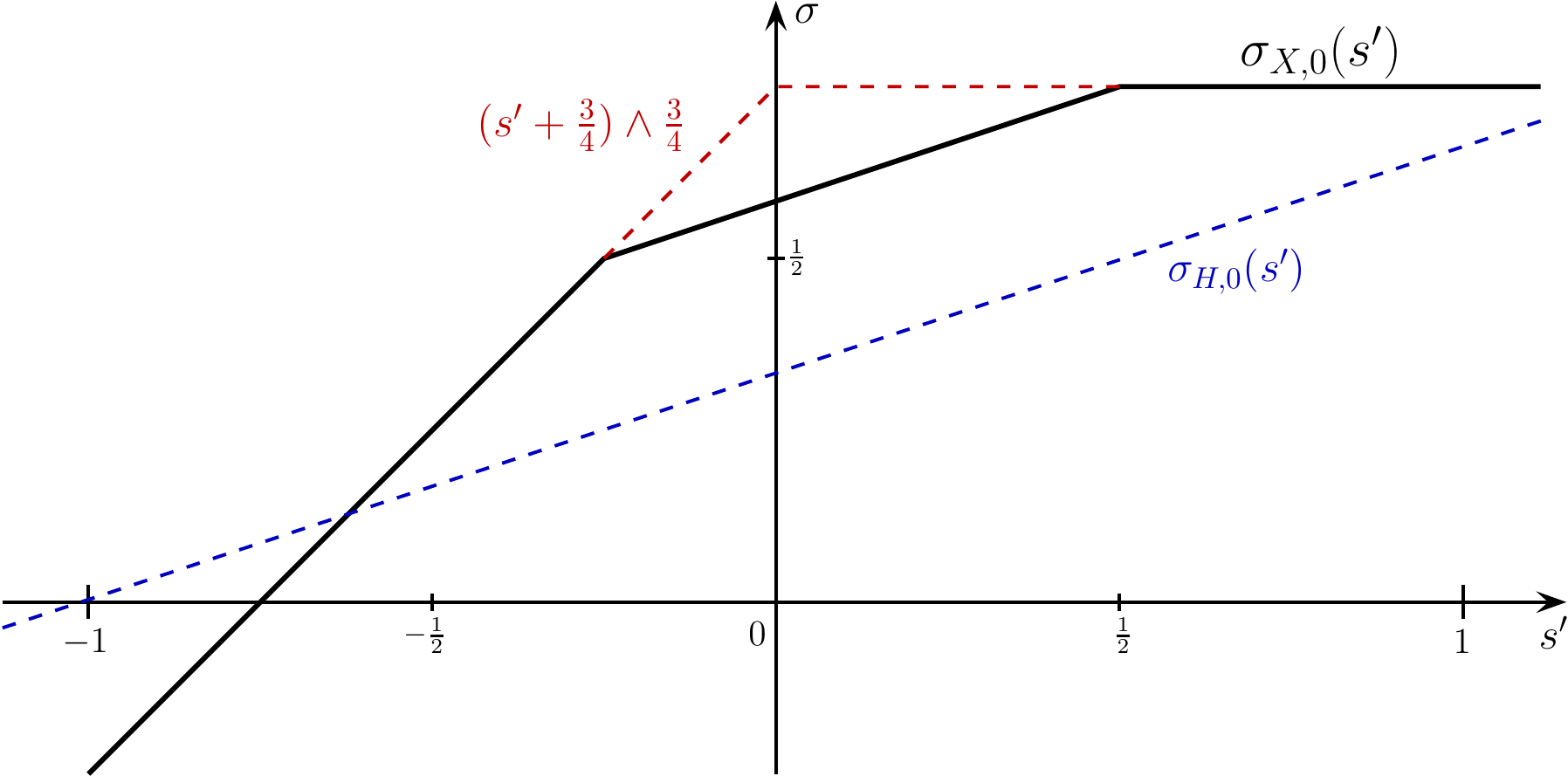}
    \caption{Example 2}
    \label{sfig:2ml_mbm_ex2}
  \end{subfigure}
  \caption{2-microlocal frontiers of multifractional Brownian motion}
  \label{fig:2ml_mbm_exs}
\end{figure}

In the following second example, we illustrate the particular Hölder regularity which may appear at $t=0$.
\begin{example} \label{ex:2ml_mbm2}
  Consider the Hurst function $H:t\mapsto \tfrac{3}{4} + t \sin\pthb{ \tfrac{1}{t^2} }$, which locally corresponds to the "chirp" function. Then, the 2-microlocal frontier of the multifractional Brownian motion at $0$ is equal to
  \[
    \forall s'\in\R;\quad \sigma_{X,0}(s') = \pthB{ s' + \frac{3}{4} }\wedge \pthB{ \frac{s'}{3} + \frac{7}{12} }\wedge\frac{3}{4}.
  \]
  Figure~\ref{sfig:2ml_mbm_ex2} gives an illustration of this particular and hybrid 2-microlocal frontier.

  To obtain this estimate, we need to refine the proof of Theorem~\ref{th:2ml_mbm}. Let first recall that the 2-microlocal frontier of $H$ at $0$ is equal to
  \[
    \forall s'\in\R;\quad \sigma_{H,0}(s') = \frac{1}{3}(s' + 1).
  \]
  If we consider the regularity of the term $\partial_H B(v,H_0) (H_v-H_v)$, we know there exists for any $\eps>0$ a sequence $u_n\rightarrow 0$ such that $\abs{\partial_H B(u_n,H_0)} \geq u_n^{H_0-\eps}$. Then, owing to form of the Hurst function, there exists $v_n\rightarrow 0$ such that $\abs{H(u_n)-H(v_n)}\geq \abs{u_n-v_n}^\sigma(u_n+v_n)^{-s'}$ where $\sigma < \sigma_{H,0}(s')$. Combining the two, we obtain that the 2-microlocal frontier of the component $\partial_H B(v,H_0) (H_v-H_v)$ is
  \[
    \forall s'\geq -\frac{7}{12};\quad \sigma(s') = \sigma_{H,0}(s'+H(0)) = \pthB{ \frac{s'}{3} + \frac{7}{12} }.
  \]
  Furthermore, we know that $m_{0,H(0)} = \infty$, inducing that the other terms in the Taylor expansion are negligible.

  We note that the 2-microlocal frontier obtained in this example is consistent with the estimate of the Hölder exponents proved in Proposition~\ref{prop:pointwise_mbm}.
\end{example}

\section{Fractal dimensions of the graph and images}  \label{sec:mbm_dimension}

In this section, we aim to extend our work on the sample path geometry of mBm by studying the fractal dimension of its graph and the images of sets. As previously, we will be particularly interested in the study of the irregular case in order to generalize the classic results of \citet{Peltier.LevyVehel-1995}.

We first need to recall a few notations on fractal dimensions. For any set $A\subset\R^d$, we respectively denote by $\dimH \,A$ and $\dimBl \,A$ (or $\dimBu \,A$) the Hausdorff and the lower (or upper) Box dimensions of $A$. In the case $\dimBl \,A=\dimBu \,A$, we denote by $\dimB \,A$ the previous quantity. The reader might refer to the book of \citet{Falconer-2003} for the precise definition and the basic properties of these fractal dimensions.

As we investigate the fractal geometry of the multifractional Brownian motion, which is a purely non-stationary process, it is more relevant to consider the localised fractal dimension. For all $A\subset\R^d$ and $x\in\R^d$, it is defined by
\[
  \dimUt{x} \,A = \lim_{\rho\rightarrow 0} \,\dimU \pthb{ A\cap B(x,\rho) },
\]
where $\dimU$ denotes any fractal dimension. The previous limit always exists as the right term is decreasing when $\rho\rightarrow 0$. We note that the map $x\mapsto\dimUt{x} \,A$ is upper semi-continuous. Finally, we also observe that the fractal dimension can always be retrieve from the localised values: $\dimU A = \sup_{x\in\R^d} \dimUt{x} A$.\vsp

In the first part of this section, we investigate the local fractal dimension of the graph of the mBm. For any function $f$, its graph on the set $V$ will be denoted $\gr(f,V)$, i.e.
\[
  \gr(f,V) = \brcb{ (t,f(t)) : t\in V }.
\]
In the case $V=\R$, we will simply use the notation $\gr(f)$. To simplify, we will write $\dimUt{t}\gr(f)$ for the local fractal dimension at $(t,f(t))$. The latter is therefore equivalently defined by
\begin{equation}  \label{eq:def_local_dim}
  \dimUt{t} \gr(f,V) \eqdef \lim_{\rho\rightarrow 0} \,\dimU\gr(f,V\cap B(t,\rho)).
\end{equation}
Several deterministic properties satisfied by the local graph dimension are presented in the Appendix~\ref{sec:mbm_appendix}, including some connections with the 2-microlocal frontier.\vsp

The graph dimension of an $\Hi_0$-multifractional Brownian motion has already been determined by \citet{Peltier.LevyVehel-1995}, proving that with probability one
\begin{equation}
  \forall t\in\R;\quad \dimHt{t} \gr(f) = \dimBt{t} \gr(f) = 2 - H(t).
\end{equation}
In the next result, we investigate the extension of this formula to the Box dimension of the graph of a general multifractional Brownian motion.
\begin{theorem}  \label{th:mbm_dim_box}
  Suppose $X$ is a multifractional Brownian motion. Then, with probability one, it satisfies
  \[
    \forall t\in\R\setminus\brc{0};\quad \dimBlt{t}\gr(X) = \pthb{ 2-H(t) } \vee \dimBlt{t}\gr(H)
  \]
  and
  \[
    \forall t\in\R\setminus\brc{0};\quad \dimBut{t}\gr(X) = \pthb{ 2-H(t) } \vee \dimBut{t}\gr(H).
  \]
  Hence, the graph has a local Box dimension at $t$ when the previous two quantities are equal.
\end{theorem}
\begin{proof}
  We first obtain the upper bound for both the lower and upper Box dimensions. Let $t\in\R$ and $\rho>0$. Owing to Lemma~\ref{prop:2ml_fGf}, with probability one, there exist $c_1(\omega)$ and $\eps>0$ such that for all $u,v\in B(t,\rho)$,
  \[
    \abs{X_u-X_v} \leq c_1 \abs{u - v}^{H(t)-\eps} + c_1 \abs{H(u) - H(v)}.
  \]
  Let $N_\delta(H)$ denotes the smallest number of balls of diameter $\delta$ necessary to cover the graph of $H$ on the interval $\ivff{t-\rho,t+\rho}$. Up to a constant, we can assume that this cover is organised on columns of size $\delta$, as usually presented in the case of a graph covering. Owing to the upper bound of the increment $X_u-X_v$, we know that by adding $\ceil{\delta^{H(t)-\eps}\,\delta^{-1}}$ balls to every column covering, we obtain a covering of the graph of the process $X$. Hence, there exists a constant $c_2$ such that
  \[
    \forall \delta>0;\quad N_\delta(X) \leq c_2\pthb{ \rho \delta^{H(t)-\eps-2} + N_\delta(H) }
  \]
  and therefore
  \begin{align*}
    \frac{\log( N_\delta(X) )}{-\log(\delta)}
    &\leq \frac{\log( c_2\rho \delta^{H(t)-\eps-2} + c_2 N_\delta(H) )}{-\log(\delta)} \\
    &\leq \max\pthbb{ \frac{\log( c_3\delta^{H(t)-\eps-2}) }{-\log(\delta)},  \frac{\log( c_3 N_\delta(H)) }{-\log(\delta)} }.
  \end{align*}
  Considering the inferior and superior limits of this inequality, we respectively obtain the upper bound of the lower and upper Box dimensions.

  To prove the lower bound $2-H(t)$, we use classic potential methods that have already been presented by \citet{Peltier.LevyVehel-1995} and \citet{Herbin.Arras.ea-2012} for the multifractional Brownian motion. Briefly, for any interval $\ivff{a,b}\subset\R_+$, there exists $C>0$ such that $\esp{X_u-X_v}^2 \geq C \abs{u-v}^{2\underline{H}}$, where $\underline{H} = \min_{u\in\ivff{a,b}} H(u)$. This estimate is sufficient, using a standard potential argument to show that almost surely,
  \[
    2 - \min_{u\in\ivff{a,b}} H(u) \leq \dimH\gr(X,\ivff{a,b}) \leq \dimBl\gr(X,\ivff{a,b}) \leq \dimBu\gr(f,\ivff{a,b}).
  \]
  Considering $a,b\in\Q$, we obtain the inequality with probability one for any rational interval. Then, we note that Equation~\eqref{eq:def_local_dim} holds as well if we use in the limit rational intervals containing $t$. Hence, owing to the continuity of $H$, with probability one and for all $t\in\R$, we get $2-H(t)\leq\dimBlt{t}\gr(X)$ and $2-H(t)\leq\dimBut{t}\gr(X)$.

  To prove the remaining lower bound, let us first consider the case $m_{t,H(t)} = 1$. The constant $\rho$ can be chosen small enough such that for all $(v,h)$ in neighbourhood of $(t,H(t))$, $\abs{\partial_H B(v,h)} > c_4$, where $c_4$ is a positive constant. Furthermore, for all $u,v\in B(t,\rho)$,
  \[
    X_u-X_v = B(u,H(u)) - B(v,H(v)) + \int_{H(v)}^{H(u)} \partial_H B(v,h) \,\dt h.
  \]
  Hence, in the neighbourhood of $t$,
  \[
    c_4 \abs{ H(u) - H(v) } \leq \absbb{ \int_{H(v)}^{H(u)} \partial_H B(v,h) \,\dt h }  \leq \abs{X_u-X_v} + c_5\abs{u-v}^{H(t)-\eps}.
  \]
  Using the covering argument described at the beginning of the proof, we obtain that the inequality $\dimBlt{t}\gr(H) \leq \pthb{ 2-H(t) } \vee \dimBlt{t}\gr(X)$ and $\dimBut{t}\gr(H) \leq \pthb{ 2-H(t) } \vee \dimBut{t}\gr(X)$. Since we already know that the lower and upper Box dimensions are greater than $2-H(t)$, it induces the expected lower bound.

  Finally, let us extend the inequality to any $t\in\R\setminus\brc{0}$ and any value of $m_{t,H(t)}$. We have already noted that the function $t\mapsto\dimUt{t}\gr(H)$ is upper semi-continuous for every dimension considered. As a consequence, there exists a deterministic countable set $E\subset\R\setminus\brc{0}$ such that for all $t$, there is a sequence $(t_n)_{n\in\N}\in E$ satisfying
  \[
    \dimUt{t}\gr(H) = \lim_{n\rightarrow\infty} \dimUt{t_n}\gr(H),
  \]
  Since for every $s\in E$, $\partial_H B(s,H(s))$ is a centered Gaussian variable with a positive variance, with probability one and for all $s\in E$, $\partial_H B(s,H(s)) \neq 0$. Hence, as the local dimension is upper semi-continuous,
  \begin{align*}
    \dimBlt{t}\gr(X)
    &\geq \limsup_{n\in\N} \,\dimBlt{t_n}\gr(X) \\
    &= \limsup_{n\in\N} \,\pthb{ 2-H(t_n) } \vee \dimBlt{t_n}\gr(H)
    = \pthb{ 2-H(t) } \vee \dimBlt{t}\gr(H).
  \end{align*}
  The argument holds as well for the upper Box dimension, therefore uniformly proving the result.
\end{proof}
Interestingly, we have obtained an expression of the Box dimension of the graph which is similar to the behaviour observed on a fractional Brownian motion with variable drift (see the recent works \cite{Charmoy.Peres.ea-2012,Peres.Sousi-2012}). This common property happens to extend as well to the Hausdorff dimension of the graph, recently obtained by \citet{Peres.Sousi-2013} on the fBm with variable drift.

The Hausdorff dimension of the graph is a bit more complicate to analyse. For this purpose, we have to introduce a more general dimension called the \emph{parabolic Hausdorff dimension}. We first need to define a \emph{parabolic metric} $\varrho_H$ on $\R^2$, with $H>0$:
\[
  \varrho_H\pthb{ (u,x) \,; (v,y) } \eqdef \max\pthb{ \abs{u-v}^{H}, \abs{x-y} }.
\]
For any set $A\subset\R^2$, we denote by $\dimH(A \,;\varrho_H)$ the \emph{parabolic Hausdorff dimension} of $A$. It is defined similarly to the classic Hausdorff dimension using covering balls relatively to the metric $\varrho_H$, i.e. it corresponds to the infimum of $s\geq 0$ for which
\[
  \lim_{\delta\rightarrow 0} \inf\brcbb{ \sum_{i=0}^\infty \diam(O_i\,; \varrho_H)^s : (O_i)_{i\in N} \text{ is a $\delta$-cover of $A$} } < \infty
\]
Note that given the form of the parabolic metric, a $\varrho_H$-ball $B((t,x),\delta)$ has the following form: $B((t,x),\delta) = \ivff{t-\delta^{1/H},t+\delta^{1/H}}\times\ivff{x-\delta,x+\delta}$.

The parabolic Hausdorff dimension has first been considered by \citet{Taylor.Watson-1985} for the study of polar sets of the heat equation. Recently, it has been proved to be useful by \citet{Khoshnevisan.Xiao-2012} to analyse the geometry of the images of Brownian motion. Finally, as previously outlined, \citet{Peres.Sousi-2013} have shown it is also a natural idea for the study of fractional Brownian motion with variable drift. Note that the convention used in the latter work is slightly different from ours.
\begin{theorem}  \label{th:mbm_dim_haus}
  Suppose $X$ is a multifractional Brownian motion. Then, with probability one, it satisfies
  \[
    \forall t\in\R\setminus\brc{0};\quad \dimHt{t}\gr(X) = 1 + H(t)\pthb{ \dimHt{t}\pthb{\gr(H) \,; \varrho_{H(t)}} - 1 }.
  \]
\end{theorem}
\begin{proof}
  To begin with, we prove that with probability one,
  \[
    \forall t\in\R\setminus\brc{0};\quad \dimHt{t}\pthb{\gr(X) \,; \varrho_{H(t)}} = \dimHt{t}\pthb{\gr(H) \,; \varrho_{H(t)}}
  \]
  The sketch of the proof of this equality is similar to Theorem~\ref{th:mbm_dim_box}, as it uses as well an ad-hoc covering argument. Let us set $\delta>0$, $\rho>0$ and $\gamma > \dimHt{t}\pthb{\gr(H) \,; \varrho_{H(t)}}$. Suppose $(O_i)_{i\in\N}$ is a $\delta$-cover of $\gr(H,\ivff{t-\rho,t+\rho})$ with of $\varrho_{H(t)}$-balls and such that
  \[
    \sum_{i=0}^\infty \delta_i^\gamma < \infty\quad\text{where }\delta_i \eqdef \diam(O_i\,; \varrho_{H(t)})\hem \forall i\in\N.
  \]
  We observe that $\abs{X_u-X_v} \leq c_1 \abs{u - v}^{H(t)-\eps} + c_1 \abs{H(u) - H(v)}$ for all $u,v\in B(t,\rho)$. For any $i\in\N$, we need at most $2\delta_i^{1-\eps/H(t)}\cdot\delta_i^{-1}$ $\varrho_{H(t)}$-balls of diameter $\delta_i$ to cover the part of the graph related to the term $\abs{u - v}^{H(t)-\eps}$. Hence, from $(O_i)_{i\in\N}$, we construct a $\delta$-cover $(V_k)_{k\in\N}$ of $\gr(X,\ivff{t-\rho,t+\rho})$ which satisfies
  \[
    \sum_{k=0}^\infty \diam(V_k\,; \varrho_{H(t)})^s \leq c_2 \sum_{i=0}^\infty \delta_i^{-\eps/H(t)} \cdot \delta_i^s.
  \]
  The last series converges when $s \geq \gamma + \eps/H(t)$. Hence, considering the limits $\eps\rightarrow 0$, $\rho\rightarrow 0$ and $\gamma \rightarrow \dimHt{t}\pthb{\gr(H) \,; \varrho_{H(t)}}$, we obtain $\dimHt{t}\pthb{\gr(X) \,; \varrho_{H(t)}} \leq \dimHt{t}\pthb{\gr(H) \,; \varrho_{H(t)}}$.\vsp

  To prove the other side inequality, we proceed similarly to the proof of Theorem~\ref{th:mbm_dim_box}. We first assume that $m_{t,H(t)} = 1$. In this case, the estimate of the increments presented in Theorem~\ref{th:mbm_dim_box} and the previous reasoning similarly imply that $\dimHt{t}\pthb{\gr(H) \,; \varrho_{H(t)}} \leq \dimHt{t}\pthb{\gr(X) \,; \varrho_{H(t)}}$. Then, we use the upper semi-continuity of the function $t\mapsto\dimHt{t}\pthb{\gr(H) \,; \varrho_{H(t)}}$, obtained in Lemma~\ref{lemma:dim_para_usc}, to extend the inequality to any $t\in\R\setminus\brc{0}$.\vsp

  The upper bound of $\dimHt{t}\gr(X)$ is then a consequence of the deterministic Lemma~\ref{lemma:dim_para_bounds}, with $H_1 = 1$ and $H_2 = H(t)$. With probability one, for all $t\in\R$, we have $\dimH \gr(X,\ivff{t-\rho,t+\rho}) \leq 1 + H(t)\pthb{ \dimH\pthb{ \gr(X,\ivff{t-\rho,t+\rho})\,;\varrho_{H(t)}} -  1 }$. Therefore, when $\rho\rightarrow 0$, we obtain the first upper bound with probability one: for all $t\in\R\setminus\brc{0}$, $\dimHt{t} \gr(X) \leq 1 + H(t)\pthb{ \dimHt{t} \pthb{ \gr(H)\,;\varrho_t} - 1 }$.\vsp

  To prove the lower bound, we make use of potential arguments on the parabolic and euclidean metrics. We first set $t\in\R\setminus\brc{0}$. Then, let $\rho>0$ and $\gamma < \dimH\pthb{ \gr(X,B(t,\rho))\,;\varrho_{H(t)}}$. Using Frostman's lemma applied on the parabolic metric $\varrho_t$, there exists a probability measure $\mu$ on $B(t,\rho)$ such that
  \[
    \Ei_\gamma(\mu_H) = \iint_{B(t,\rho)^2} \frac{ \mu(\dt u)\,\mu(\dt v) }{ \pthb{ \abs{u-v}^{H(t)} +\abs{ H(u)-H(v) } }^\gamma } < \infty.
  \]
  Suppose $\mu_X$ denotes the image of the measure $\mu$ by the map $u\mapsto(u,X_u)$. To prove the lower on the dimension of the graph of $X$, we need to study $s$-energy of $\mu_X$, i.e.
  \begin{align*}
    \Ei_s(\mu_X)
    =\iint_{\R^4} \frac{\mu_X(\dt x)\,\mu_X(\dt y)}{ \norm{x-y}^s }
    = \iint_{B(t,\rho)^2} \frac{\mu(\dt u)\,\mu(\dt v)}{ \pthb{ \abs{u-v}^{2} +\abs{ X_u-X_v }^2 }^{-s/2} },
  \end{align*}
  where $s > 1$. Let us set $\eps>0$. To study the previous energy, we need to split the integral into two components, respectively denoted $\Ei^1_s$ and $\Ei^2_s$:
  \begin{align*}
    \Ei^1_s + \Ei^2_s = \iint_{B(t,\rho)^2} \frac{\mu(\dt u)\,\mu(\dt v)}{ \pthb{ \abs{u-v}^{2} +\abs{ X_u-X_v }^2 }^{-s/2} }\pthB{ &\indi_{\brcb{\abs{H(u)-H(v)}\leq\abs{u-v}^{H(t)(1-\eps)}}} \\
    + &\indi_{\brcb{\abs{H(u)-H(v)}>\abs{u-v}^{H(t)(1-\eps)}}} }.
  \end{align*}
  Let us prove $\Ei^1_s$ is finite. For this purpose, we estimate its expectation
  \begin{align*}
    \esp{ \Ei^1_s } = \iint_{B(t,\rho)^2} \espB{ \pthb{ \abs{u-v}^{2} +\abs{ X_u-X_v }^2 }^{-s/2} }\indi_{\brcb{\abs{H(u)-H(v)}\leq\abs{u-v}^{H(t)(1-\eps)}}} \mu(\dt u)\,\mu(\dt v).
  \end{align*}
  Using the notation $\sigma_{u,v}^2 \eqdef \esp{X_u-X_v}^2$, the inner term is equal to
  \begin{align*}
    \espB{ \pthb{ \abs{u-v}^{2} +\abs{ X_u-X_v }^2 }^{-s/2} }
    &= \frac{2}{\sigma_{u,v}\sqrt{2\pi}} \int_{\R_+} \pthb{ x^2 +\abs{u-v}^2 }^{-s/2} \exp\pthbb{ -\frac{x^2}{2\sigma^2_{u,v}} } \,\dt x \\
    &= c_1 \int_{\R_+} \pthb{ r \sigma_{u,v}^2 +\abs{u-v}^2 }^{-s/2} r^{-1/2} \e^{-r/2} \,\dt r.
  \end{align*}
  Splitting the integral into two terms, and using classic approximations (see e.g. \cite{Falconer-2003}), the previous term is upper bounded by
  \begin{align*}
    \int_0^{\abs{u-v}^2/\sigma_{u,v}^2} \abs{u-v}^{-s} r^{-1/2} \,\dt r
    + \int_{\abs{u-v}^2/\sigma_{u,v}^2}^\infty r^{-s/2} \sigma_{u,v}^{-s}  r^{-1/2} \,\dt r
    \leq c_2 \abs{u-v}^{1-s} \sigma_{u,v}^{-1}.
  \end{align*}
  We recall that \citet{Herbin-2006} has proved that
  \[
    \sigma_{u,v}^2 = \esp{X_u-X_v}^2 \asymp \abs{u-v}^{2H(t)} + \abs{H(u)-H(v)}^2
  \]
  in the neighbourhood of $t$. Hence, on the domain $\brcb{\abs{H(u)-H(v)}\leq\abs{u-v}^{H(t)(1-\eps)}}$, we observe that
  \begin{align*}
    \abs{u-v}
    &=\pthb{ \abs{u-v}^{2H(t)(1-\eps)} }^{1/2H(t)(1-\eps)} \\
    &\geq c_3 \pthb{ \abs{u-v}^{2H(t)(1-\eps)} + \abs{H(u)-H(v)}^2 }^{1/2H(t)(1-\eps)} \\
    &\geq c_4 \,\sigma_{u,v}^{1/H(t)(1-\eps)}.
  \end{align*}
  The expectation of $\Ei^1_s$ is therefore upper bounded by
  \begin{align*}
    \esp{ \Ei^1_s }
    &\leq c_4 \iint_{B(t,\rho)^2} \sigma_{u,v}^{-1+(1-s)/H(t)(1-\eps)} \,\mu(\dt u)\,\mu(\dt v) \\
    &\leq c_5 \iint_{B(t,\rho)^2} \pthb{ \abs{u-v}^{H(t)} +\abs{ H(u)-H(v) } }^{-1+(1-s)/H(t)(1-\eps)}\,\mu(\dt u)\,\mu(\dt v).
  \end{align*}
  The term $\esp{ \Ei^1_s }$ is thus finite when $-1+(1-s)/H(t)(1-\eps) > - \gamma$, i.e. when
  \[
    s < 1 + (\gamma-1)\pthb{H(t)(1-\eps) }\quad\text{where } \gamma < \dimH\pthb{ \gr(X,B(t,\rho))\,;\varrho_{H(t)}}.
  \]
  Hence, considering the limits on $\gamma$ and $\eps$, the upper bound obtained on $s$ corresponds to the expected estimate.

  Let now consider the second term $\Ei^2_s$. We know that $m_{t,H(t)}=1$ almost surely. Recall that the increment $X_u - X_v$ has the following form
  \[
    X_u-X_v = B(u,H(u)) - B(v,H(u)) + \int_{H(v)}^{H(u)} \partial_H B(v,h) \,\dt h.
  \]
  If $\rho$ is sufficiently small, there exists a positive constant $c_1 > 0$ such that for all $s\in B(t,\rho)$ and any $h\in\ivff{\min_{B(t,\rho)} H(u),\max_{B(t,\rho)} H(u)}$, we have $\abs{\partial_H B(v,h)} \geq c_1$.
  Furthermore, there also exists $c_2 > 0$ such that
  \[
    \forall u,v\in B(t,\rho);\quad \absb{B(u,H(u)) - B(v,H(u))} \leq c_2 \abs{u-v}^{H(t)(1-\eps/2)}.
  \]
  Using the previous estimates,
  \begin{align*}
    \abs{X_u-X_v}
    &\geq \absbb{ \int_{H(v)}^{H(u)} \partial_H B(v,h) \,\dt h } - \absb{B(u,H(u)) - B(v,H(u))} \\
    &\geq c_1 \abs{ H(u)-H(v) } \pthb{ 1 - c_2\abs{ H(u)-H(v) }^{-1}\cdot\abs{u-v}^{H(t)(1-\eps/2)} }
  \end{align*}
  Hence, on the domain $\brcb{\abs{H(u)-H(v)}>\abs{u-v}^{H(t)(1-\eps)}}$, we obtain
  \begin{align*}
    \abs{X_u-X_v}
    \geq c_1 \abs{ H(u)-H(v) } \pthb{ 1 - c_2 \abs{u-v}^{H(t)\eps/2} }
    \geq c_3 \abs{ H(u)-H(v) },
  \end{align*}
  when $\rho$ is sufficiently small. Therefore, the energy term $\Ei^2_s$ satisfies
  \begin{align*}
    \Ei^2_s
    &\leq c_4 \iint_{B(t,\rho)} \abs{H(u)-H(v)}^{-s} \indi_{\brc{\abs{H(u)-H(v)}>\abs{u-v}^{H(t)(1-\eps)}} }\,\mu(\dt u)\,\mu(\dt v) \\
    &\leq c_5 \iint_{B(t,\rho)^2} \pthb{ \abs{u-v}^{H(t)} +\abs{ H(u)-H(v) } }^{-s}\,\mu(\dt u)\,\mu(\dt v).
  \end{align*}
  The last integral is finite when $s < \gamma$. But, since $\gamma > 1$, we remark
  that $\gamma(1-H(t)) > 1-H(t)$ and thus $\gamma > 1 + H(t)(\gamma-1)$, showing that $\Ei^2_s$ is always finite when $s < 1 + (\gamma-1)H(t)$. Therefore, we have proved that the $s$-energy $\Ei_s(\mu_X)$ is almost surely finite for all $s < 1 + (\gamma-1)\pthb{H(t)(1-\eps)}$. Considering the limits $\eps\rightarrow 0$ and $\gamma \rightarrow \dimHt{t}\pthb{ \gr(X)\,;\varrho_{H(t)}}$, we obtain
  \[
    \dimHt{t}\gr(X) \geq 1 + H(t)\pthb{ \dimHt{t}\pthb{\gr(H) \,; \varrho_{H(t)}} - 1 } \hem\text{a.s.}
  \]
  The lower bound has been obtained for a fixed $t\in\R\setminus\brc{0}$. To extend the latter uniformly on $\R\setminus\brc{0}$, we proceed similarly to Theorem~\ref{th:mbm_dim_box}. Shortly, as proved in Lemma~\ref{lemma:dim_para_usc}, the map $t\mapsto\dimHt{t}\pthb{\gr(H) \,; \varrho_{H(t)}}$ is upper semi-continuous. Hence, there exists a countable set $E$ such that for any $t$, there is sequence $(t_n)_{n\in\N}\in E$ satisfying
  \[
    \dimHt{t}\pthb{\gr(H) \,; \varrho_{H(t)}} = \lim_{n\rightarrow\infty} \dimHt{t_n}\pthb{\gr(H) \,; \varrho_{H(t_n)}}.
  \]
  The equality on the dimension of the graph holds almost surely for all $s\in E$. Therefore, with probability one, for all $t\in\setminus\brc{0}$,
  \begin{align*}
    \dimHt{t}\gr(X)
    &\geq \limsup_{n\in\N} \dimHt{t_n}\gr(X) \\
    &= 1 + \limsup_{n\in\N} \,H(t_n)\pthb{ \dimHt{t_n}\pthb{\gr(H) \,; \varrho_{H(t_n)}} - 1 } \\
    &= 1 + H(t)\pthb{ \dimHt{t}\pthb{\gr(H) \,; \varrho_{H(t)}} - 1 }.
  \end{align*}
  This last inequality concludes the proof, since the upper bound is already uniform on $\R$.
\end{proof}
Note that Theorem~\ref{th:mbm_dim_haus} does not contradict the result obtained by \citet{Peltier.LevyVehel-1995} on $\Hi_0$-mBms. Indeed, under this assumption $\Hi_0$ and due to the form of the parabolic metric $\varrho_{H(t)}$, covering the graph of $H$ in the neighbourhood of $t$ is equivalent to covering a straight line, implying that the local dimension satisfies: $\dimHt{t}\pthb{\gr(H) \,; \varrho_{H(t)}} = H(t)^{-1}$. Hence, the application of Theorem~\ref{th:mbm_dim_haus} proves that with probability one
\[
  \forall t\in\R\setminus\brc{0};\quad \dimHt{t} \gr{X} = 2 - H(t),
\]
where $X$ is an $\Hi_0$-multifractional Brownian motion.

It is interesting to observe that the Hausdorff and Box dimensions of the graph of  mBm and fBm with variable drift have a similar form despite the rather different Gaussian structures of these two processes. Indeed, the first one is a centered process whose incremental variance satisfies
\[
  \esp{X_u-X_v}^2 \asymp \abs{u-v}^{2H(t)} + \abs{H(u)-H(v)}^2,
\]
whereas the latter still has the covariance of fBm, but a mean equal to the drift $f(t)$.

As a consequence of these similarities, the examples constructed by \citet{Peres.Sousi-2013} can also be applied to the mBm, proving that the graph of the latter may have different Hausdorff and Box dimensions.
\begin{example}  \label{ex:mbm_gr_dim1}
  \citet{Peres.Sousi-2013} have computed the parabolic Hausdorff dimension of the celebrated deterministic construction of \citet{McMullen-1984}. Hence, there exists a continuous function $H:\R\mapsto\ivfo{1/2,1}$ such that for all $t\in\R$, $\widetilde\alpha_{X,t} = \log(2) / \log(3) \eqdef \theta$,
  \[
    \dimHt{t}\pthb{\gr(H) \,; \varrho_{1/2}} = \log_2\pthb{5^{2\theta}+1}\hem\text{and}\hem \dimBt{t}\gr(H) = 2 - \theta.
  \]
  Then, considering the mBm $X$ with Hurst function $H$, we first observe that with probability one, for all $t>0$, $\widetilde\alpha_{X,t} = \theta$. Furthermore, as a consequence of the properties of $H$, at every $t>0$ such that $H(t) = 1/2$,
  \begin{equation*}  \label{eq:mbm_gr_dim1}
    \dimHt{t}\gr(X) = \frac{1+\log_2\pthb{5^{2\theta}+1}}{2} < \dimBt{t}\gr(X) = 2 - \theta.
  \end{equation*}
  Hence, at these particular times, the local Hausdorff dimension is strictly smaller that the Box one. In addition, as shown by \citet{Peres.Sousi-2013}, the local Hausdorff dimension is such that $\dimHt{t}\gr(X) > \max\brcb{ \dimHt{t}\gr(H), 3/2 }$, meaning the simple and straightforward lower bound on the Hausdorff dimension is not optimal.
\end{example}
This example illustrates the kind of unusual fractal geometry the mBm can present. It clearly shows that the parabolic Hausdorff dimension is the relevant concept, as in this case, $\dimHt{t}\gr(H)$ happens to be insufficient to characterise entirely the Hausdorff dimension of the graph. We also note that in contrary to fBm, the fractal dimension of the graph can be completely unrelated to the local Hölder exponent.

As presented in the following second example, more classic sample path properties also occur.
\begin{example}  \label{ex:mbm_gr_dim2}
  Suppose $H$ is a sample path of a fractional Brownian motion with Hurst exponent $\alpha$. Up to a rescaling, it is assumed to be valued in the interval $\ivoo{0,1}$. Then, using some classic estimates (see e.g. \cite{Falconer-2003}), we easily obtain that
  \[
    \forall t\in\R;\quad \dimHt{t}\pthb{\gr(H) \,; \varrho_{H(t)}} = \max\brcbb{ \frac{1}{H(t)}, 1 + \frac{1-\alpha}{H(t)} }.
  \]
  Then, owing to Theorem~\ref{th:mbm_dim_haus}, a mBm $X$ with Hurst function $H$ satisfies with probability one:
  \[
    \forall t\in\R\setminus\brc{0};\quad \dimHt{t} \gr(X) = \dimBt{t} \gr(X) = \max\brcb{ 2-\alpha,2-H(t) } = 2 - \widetilde\alpha_{X,t}.
  \]
\end{example}
This example displays a more common fractal geometry, where the Hausdorff and Box dimensions are both equal to $2 - \widetilde\alpha_{X,t}$. Nevertheless, the local Hölder exponent may still have an irregular behaviour, depending on the value of $H(t)$.\vsp

In the second part of this section, we are interested in studying the fractal dimension of an image $X(F)$, where $F$ can be any fractal set.
A large literature has already investigated the question of the geometry of images of Gaussian processes. It is a well-known result (see e.g. \cite{Kahane-1985}) that for any fixed Borel set $F\in\ivoo{0,\infty}$,
\begin{align}  \label{eq:fbm_images}
  \dimH \,B^H(F) = \min\brcbb{1,\frac{\dimH F}{H}}\quad\text{a.s.}
\end{align}
\citet{Mountford-1989a,Khoshnevisan.Wu.ea-2006,Wu.Xiao-2007a} have generalized the previous result to $(N,d)$-(fractional) Brownian sheet, and in particular, obtained uniform results on $F$ under some assumption on the dimensions $N$ and $d$.

Adapting classic arguments from \citet{Kahane-1985}, we can easily extend Equation~\eqref{eq:fbm_images} to an $\Hi_0$-mBm, i.e. for any fixed Borel set $F\in\ivoo{0,\infty}$ and all $t\in\ivoo{0,\infty}$
\begin{align}  \label{eq:mbm_images_H0}
  \dimHt{X(t)} \,X(F) = \min\brcbb{1,\frac{\dimHt{t} \,F}{H(t)}}\quad\text{a.s.}
\end{align}

In the irregular case, the statement is not as simple and straightforward. Furthermore, the \emph{parabolic Hausdorff dimension} again happens to be the relevant tool to describe the geometry of $X(F)$.
\begin{theorem}  \label{th:mbm_dim_images}
  Suppose $X$ is a multifractional Brownian motion. Then, with probability one,
  \[
    \forall t\in\R\setminus\brc{0};\quad \dimHt{X(t)} \,X(F) = \min\brcB{ 1, \dimHt{t} \pthb{ \gr\pth{ H,F }\,; \varrho_t } }.
  \]
\end{theorem}
\begin{proof}
  The proof of this result is similar to the work of \citet{Peres.Sousi-2013} on the drifted fBm. Let us first remark that
  \[
    \dimHt{X(t)} \,X(F) \leq \dimHt{t} \gr(X,F).
  \]
  This inequality is a simple consequence of the property of the Hausdorff under the application of Lipschitz function (here, the projection $p:(t,X(t))\mapsto X(t)$).
  Then, using Lemma 2.1 from \cite{Peres.Sousi-2013}, we obtain $\dimHt{t} \gr(X,F) \leq \dimHt{t} \pthb{ \gr\pth{ H,F }\,; \varrho_t }$, thus proving the upper bound.

  To obtain the lower bound, we use a classic potential argument. Owing to Frostman's lemma, for any $\rho>0$ and any $\gamma < \dimH \pthb{ \gr\pth{ H,F\cap B(t,\rho) } \,;\varrho_{H(t)}}$, there exists a probability measure $\mu$ on $B(t,\rho)$ such that
  \[
    \iint_{B(t,\rho)^2} \frac{ \mu(\dt u)\,\mu(\dt v) }{ \pthb{ \abs{u-v}^{H(t)} +\abs{ H(u)-H(v) } }^\gamma } < \infty.
  \]
  Recall, as proved by \citet{Herbin-2006} the covariance of the mBm has the following form
  \[
    \sigma_{u,v}^2 \eqdef \esp{X_u-X_v}^2 \asymp \abs{u-v}^{2H(t)} + \abs{H(u)-H(v)}^2
  \]
  Suppose $\mu_X$ denotes the image of the measure $\mu$ under the map $u\mapsto X_u$. Then, for any $s\in\ivoo{0,1}$, the expectation of its $s$-energy is equal to
  \begin{align*}
    \espBB{ \iint_{\R^2} \frac{\mu_X(\dt x)\,\mu_X(\dt y)}{ \norm{x-y}^s } }
    &= \iint_{\R^2} \espb{ \abs{ X_u-X_v }^{-s} } \mu(\dt u)\,\mu(\dt v) \\
    &= c_2 \iint_{\R^2} \sigma_{u,v}^{-2s} \,\mu(\dt u)\,\mu(\dt v) \\
  \end{align*}
  Therefore, owing to the form of the covariance, the integral is finite when $s\leq\gamma$. Hence, considering the limit $\rho\rightarrow 0$, we obtain $\dimHt{X(t)} \,X(F) \geq \dimHt{t} \pthb{ \gr\pth{ H,F }\,; \varrho_t }$, which concludes the proof.
\end{proof}

The local parabolic Hausdorff dimension $\dimHt{t} \pthb{ \gr\pth{ H,F }\,; \varrho_t }$ might not seem to be a very intuitive notion at first. To have a better intuition of the result described in Theorem~\ref{th:mbm_dim_images}, we may first note, similarly to Theorem~\ref{th:mbm_dim_haus}, it does not contradict Equation~\eqref{eq:mbm_images_H0} in the case of an $\Hi_0$-multifractional Brownian motion. Indeed, when $H(t) < \widetilde\alpha_{H,t}$, due to the property of the parabolic metric $\varrho_{H(t)}$, it is completely equivalent to cover $F$ and $\gr(H,F)$ in the neighbourhood of $t$. Hence, in this case
\begin{align*}
  \dimHt{t} \pthb{ \gr\pth{ H,F }\,; \varrho_t }
  &= \lim_{\rho\rightarrow 0} \dimH \pthb{ \gr\pth{ H,F\cap B(t,\rho) }\,; \varrho_t } \\
  &= \lim_{\rho\rightarrow 0} H(t)^{-1} \,\dimH \pthb{ F\cap B(t,\rho) }
  = H(t)^{-1}\, \dimHt{t} \,F.
\end{align*}
The next example shows that different behaviours may occur when the mBm is irregular.
\begin{example}  \label{ex:mbm_dim_images}
  Let $H$ be a trajectory of fractional Brownian motion defined on the interval $\ivff{0,1}$, with Hurst index $\alpha\in\ivoo{0,1}$. Up to a rescaling and a translation, the sample path satisfies $H(\ivff{0,1})\subset\ivoo{0,1}$.\vsp

  We first assume that $\alpha < H(t)$. Owing to Equation~\eqref{eq:fbm_images}, there exists a Borel set $F$ such that $\dimHt{H(t)} \,H(F) = \alpha^{-1}\dimHt{t}\,F$.
  Furthermore, since the proof of this statement is based on Frostman's lemma, there exists a probability measure on $F\cap B(t,\rho)$ such that $\iint_{\R^2} \abs{H(u)-H(v)}^{-\gamma} \mu(\dt u)\mu(\dt v) < \infty$ where $\gamma < \dimHt{H(t)} \,H(F)$. Therefore, still using Frostman's lemma, we obtain that  $\alpha^{-1}\dimHt{t}\,F  \leq \dimHt{t} \pthb{ \gr\pth{ H,F }\,; \varrho_t }$. The other side inequality is simply obtained by observing that $\widetilde\alpha_{X,t} = \alpha$, inducing the upper bound using classic covering arguments. Hence, we obtain
  \begin{equation} \label{eq:ex1_mbm_images}
    \dimHt{X(t)} \,X(F) = \dimHt{H(t)} \,H(F) = \frac{1}{\alpha} \dimHt{t} \,F\quad\text{a.s.}
  \end{equation}

  For the second example, we consider the set $F = H^{-1}\brcb{\tfrac{1}{2}}$, which might be assumed to be non-empty. \citet{Monrad.Pitt-1987} have proved that $\dimH\,F = 1 - \alpha$. In this case, as the Hurst function $H$ is constant on the set $F$, we easily prove that the dimension of $\gr(H,F)$ is characterised by the dimension on $F$. Therefore,
  \begin{equation} \label{eq:ex2_mbm_images}
    \dimHt{X(t)} \,X(F) = \dimHt{t} \pthb{ \gr\pth{ H,F }\,; \varrho_t } = \frac{1}{H(t)} \dimHt{t}\, F = \frac{1-\alpha}{H(t)} \quad\text{a.s.}
  \end{equation}

  The two previous examples show that the estimate obtained in Theorem~\ref{th:mbm_dim_images} can lead to very different results, depending on the properties of the set $F$ chosen. In particular, we observe that the Hausdorff geometry of images displays a richer structure than the local H\"older regularity. Indeed, in the irregular case, we know that the later only depends on $H(\cdot)$ sample path properties, whereas Equations \eqref{eq:ex1_mbm_images} and \eqref{eq:ex2_mbm_images} show a more complex geometry of the image $X(F)$. In the first case, we observe that the formula depends only on the local exponent of $H$ and the Hausdorff dimension of $F$, whereas in the second one, Equation \eqref{eq:ex2_mbm_images} presents a case where the Hausdorff dimension follows the classic formula \eqref{eq:fbm_images}, even though the mBm is irregular at $t$.
\end{example}

To conclude this section, we discuss the question of obtaining the Hausdorff dimension of the level sets $X^{-1}(\brc{x})$. This problem has already been investigated by \citet{Boufoussi.Dozzi.ea-2006} for the $\Hi_0$-multifractional Brownian motion, proving in particular that for all $t\in\R\setminus\brc{0}$,
\[
  \dimHt{t} \,X^{-1}(\brc{ X_t }) = 1 - H(t)\quad \text{a.s.}
\]
It is quite natural to wonder if the previous result can be extended to the irregular mBm. The upper bound can easily extended, since it is a direct consequence of the Hölder regularity of the process at $t$. In addition, we note that the proof from \cite{Boufoussi.Dozzi.ea-2006} of the lower bound stands as well in the irregular case. As a consequence, we easily obtain bounds for all $t\in\R\setminus\brc{0}$,
\[
  1- H(t) \,\leq\, \dimHt{t} \,X^{-1}(\brc{ X_t }) \,\leq\, 1 - \min\pthb{ H(t),\widetilde\alpha_{H,t} } \quad \text{a.s.}
\]
Lower and upper bounds are not equal in the irregular case, and the remaining question is whether the lower one can be improved. The sketch of the proof of this inequality is rather classic and relies on the study of the Hölder continuity of the local time (see e.g. the seminal works of \citet{Berman-1973} and \citet{Geman.Horowitz-1980} on the subject). The key point to prove the Hölder regularity is a property called the \emph{local nondeterminism (LND)}. In the case of the mBm, it states that for every $n\in\N$, there exists $c_n>0$ such that for all $(t_1,\dotsc,t_n,t)\in\ivoo{\eps,\infty}^{n+1}$
\[
  \varc{X_t}{X_{t_1},\dotsc,X_{t_n}} \geq c_n \min_{1\leq k \leq n} \abs{t-t_k}^{2H(t)}.
\]
To obtain a better lower bound for the local dimension of the level set, we would need to improve the exponent in the right term. Unfortunately, the next lemma proves that it is optimal under a mild assumption on the Hurst function.
\begin{lemma}
  Let $t\in\R$ and suppose there exists a sequence $s_i\rightarrow t$ such that $H(s_i) = H(t)$.
  Then, for any $n\in\N$, any $\eps>0$ and any $c_n>0$, there exists $(t_1,\dotsc,t_n)\in\R^n$ such that
  \[
     \varc{X_t}{X_{t_1},\dotsc,X_{t_n}} \leq c_n \min_{1\leq k \leq n} \abs{t-t_k}^{2H(t)-\eps}.
  \]
\end{lemma}
\begin{proof}
  For all $(t_1,\dotsc,t_n)\in\R^n$ and any $k\in\brc{1,\dotsc,n}$, we note that
  $\varc{X_t}{X_{t_1},\dotsc,X_{t_n}} \leq \varc{X_t}{X_{t_k}} \leq \espb{(X_t-X_{t_k})^2}$.
  Using elements $t_k\in\brc{s_i}_{i\in\N}$ and the estimate of the variance of the increments, we obtain that
  \[
    \varc{X_t}{X_{t_1},\dotsc,X_{t_n}} \leq C \min_{1\leq k \leq n} \abs{t-t_k}^{2H(t)}.
  \]
  Let set $\rho>0$. We may assume that $t_k\in B(t,\rho)$ for every $k\in\brc{1,\dots,n}$, implying that
  \[
    \varc{X_t}{X_{t_1},\dotsc,X_{t_n}} \leq C \rho^\eps \min_{1\leq k \leq n} \abs{t-t_k}^{2H(t)-\eps}.
  \]
  Using $\rho$ sufficiently small, it proves that the property of local nondeterminism does not stand with the exponent $2H(t)-\eps$, for any $\eps>0$.
\end{proof}
Hence, the classic exponent $2H(t)$ in the LND property of the mBm is most of the time optimal. Intuitively, it can be understood as following: in the neighbourhood of $t$, the Hurst function can influence the Hölder regularity, but is still a deterministic component and thus does not bring any randomness which could increase the conditional variance. As the consequence, the latter is only a result of the local fBm form of the multifractional Brownian motion.

Therefore, the estimate of the Hausdorff dimension of the level sets can not be improved using the classic methods based on the LND property, and it remains an open question to determine the precise form of the fractal dimension in terms of the geometric properties of the Hurst function $H$.

\section{Appendix: a few deterministic properties}  \label{sec:mbm_appendix}

We gather in the appendix a few deterministic properties related to 2-microlocal analysis and local graph dimensions and which have been used through the paper.
We begin with a simple upper bound on the 2-microlocal which extends the classic property of lower semi-continuity on the local Hölder exponent.
\begin{lemma}  \label{lemma:2ml_holder_bound}
  Suppose $f$ is a continuous function and $t\in\R$ such that $\liminf_{u\rightarrow t} \widetilde\alpha_{f,t} < 1$. Then, the 2-microlocal frontier of $f$ at $t$ satisfies
  \[
    \forall s'\in\R;\quad \sigma_{f,t}(s') \leq \liminf_{u\rightarrow t} \,\widetilde\alpha_{f,u}.
  \]
\end{lemma}
\begin{proof}
  Let us set $\rho>0$ and $\eps>0$. There exists $s\in B(t,\rho)$ such that $\widetilde\alpha_{f,s} \leq \liminf_{u\rightarrow t} \widetilde\alpha_{f,u} + \eps$. Then, we can find $u_n,v_n\rightarrow s$ such that
  \[
    \abs{ f(u_n) -f(v_n) } \geq n \,\abs{u_n-v_n}^{\widetilde\alpha_{f,s}+\eps}.
  \]
  Therefore, for any $s'\in\R$, we have
  \[
    \frac{ \abs{ f(u_n) - f(v_n) } }{\abs{u_n-v_n}^{\widetilde\alpha_{f,s}+\eps}\pthb{\abs{u_n-t}+\abs{v_n-t}}^{-s'}} \geq \frac{ n }{\pthb{\abs{u_n-t}+\abs{v_n-t}}^{-s'}} \longrightarrow_{n\rightarrow+\infty} +\infty,
  \]
  proving that $\sigma_{f,t}(s') \leq \liminf_{u\rightarrow t} \widetilde\alpha_{f,u} + \eps$, for any $\eps>0$.
\end{proof}

In the next two lemmas, we present an extension of a well-known property connecting the local dimension of the graph to the local Hölder exponent: $\dimBut{t} \gr(f) \leq 2-\widetilde\alpha_{f,t}$.
\begin{lemma}  \label{lemma:2ml_dimB}
  Suppose $f$ is a continuous function. Then, for all $t\in\R$,
  \[
    \dimBlt{t}\gr(f) \leq \dimBut{t}\gr(f) \leq 2 - \pthb{ \sigma_{f,t}(1)\wedge 1 }.
  \]
\end{lemma}
\begin{proof}
  We need to construct a local cover of the graph which is sufficiently efficient.
  Let first set $t\in\R$ and $\sigma < \sigma_{f,t}(1)\wedge 1$. Then, there exists $\rho>0$ and $C>0$ such that for all $u,v\in B(t,\rho)$,
  \[
    \abs{ f(u) - f(v) } \leq C \abs{u-v}^\sigma \pthb{ \abs{u-t} + \abs{v-t} }^{-1}.
  \]
  The cover is constructed similarly to the original proof. Hence, if $N_\delta$ denotes the number of $\delta$-boxes necessary to cover $\gr(f,\ivff{t-\rho,t+\rho})$, we have
  \begin{align*}
    N_\delta
    &\leq C \delta^{-1} + \sum_{k=-\ceil{\rho/\delta}}^{\ceil{\rho/\delta}} \sup_{u,v\in t+\ivff{k\delta,(k+1)\delta}} \abs{f(u)-f(v)} \cdot \delta^{-1} \\
    &\leq C \delta^{-1} + \sum_{k=-\ceil{\rho/\delta}}^{\ceil{\rho/\delta}} C \delta^\sigma (k\delta)^{-1}\cdot \delta^{-1}  \\
    &= C \delta^{-1} + 2C \delta^{\sigma-2} \sum_{k=1}^{\ceil{\rho/\delta}} k^{-1}
    \leq 2C \,\delta^{\sigma-2} \pthb{ 1 + \log\ceil{\rho/\delta} }.
  \end{align*}
  since $2-\sigma > 1$. Then,
  \[
    \dimBu \gr(f,\ivff{t-\rho,t+\rho}) = \limsup_{\delta\rightarrow 0} \frac{N_\delta}{-\log\delta}\leq 2-\sigma,
  \]
  for all $\sigma < \sigma_{f,t_0}(1)\wedge 1$, which proves to the expected inequality.
\end{proof}

\begin{lemma}  \label{lemma:2ml_dimH}
  Suppose $f$ is a continuous function. Then, for all $t\in\R$
  \[
    \dimHt{t}\gr(f) \leq 2 - \pthb{ \sigma_{f,t}(+\infty)\wedge 1 }.
  \]
\end{lemma}
\begin{proof}
  As the estimate we want to obtain differs from the previous lemma, we need to adopt a constructing procedure which is slightly different. Let $t\in\R$ and $\sigma < \sigma_{f,t}(+\infty)\wedge 1$. Then, there exist $s'\geq 0$, $\rho>0$ and $C>0$ such that for all $u,v\in B(t,\rho)$
  \[
    \abs{ f(u) - f(v) } \leq C\abs{u-v}^\sigma \pthb{ \abs{u-t} + \abs{v-t} }^{-s'}.
  \]
  Let us set $\delta>0$ and $\gamma>1$. The $\delta$-cover of $\gr(f,\ivff{t-\rho,t+\rho})$ is constructed as following. There exists $n\in\N$ such that for all $k\geq n$, $k^{-\gamma} < \delta$. To cover the set $\ivff{t-\rho,t+\rho}$, we first choose successive intervals $(I_k)_{k\geq n}$ of size $k^{-\gamma}$ in the neighbourhood of $t$. The rest is simply covered by intervals of size $\delta$. The first type of intervals cover a set whose length is equal to:
  \[
    \sum_{k=n}^{+\infty} k^{-\gamma} \geq \int_n^{+\infty} x^{-\gamma} \dt x = C_\gamma n^{1-\gamma}.
  \]
  Hence, the number $M_\delta$ of intervals of size $\delta$ which are necessary is upper bounded by
  \[
    M_\delta \leq \delta^{-1} \pthb{ 2\rho -  C_\gamma n^{1-\gamma} } .
  \]
  Let now consider the cover $(O_i)_{i\in\N}$ of $\gr(f,\ivff{t-\rho,t+\rho})$ which can be defined from the previous intervals (still using the same idea of covering by columns). It satisfies for any $s>0$
  \begin{align*}
    \sum_{i\in\N} \abs{O_i}^s \leq \sum_{k=n}^{+\infty } (k^{-\gamma})^s \cdot k^{\gamma} \,\sup_{u,v\in I_k} \abs{f(u)-f(v)} + \sum_{k=1}^{M_\delta} \delta^s\cdot \delta^{-1} \sup_{u,v\in \widetilde I_k} \abs{f(u)-f(v)}.
  \end{align*}
  Let us estimate the first term,
  \begin{align*}
    \sum_{k=n}^{+\infty } (k^{-\gamma})^s \cdot k^{\gamma} \,\sup_{u,v\in I_k} \abs{f(u)-f(v)}
    &\leq \sum_{k=n}^{+\infty } \pth{ k^{-\gamma} }^{\sigma+s-1} \abs{ x_k - t }^{-s'} \\
    &\leq C_\gamma \sum_{k=n}^{+\infty } \pth{ k^{-1} }^{\gamma(\sigma+s-1)-s'(\gamma-1)}. \\
  \end{align*}
  where we have used that $d(t,I_k)\sim k^{1-\gamma}$.
  The previous series converges if $s$ satisfies $\gamma(\sigma + s - 1) - s'(\gamma-1) > 1$, i.e. if $s > 1 -\sigma +s'(1-\gamma^{-1}) + \gamma^{-1}$.

  Let now consider the second term,
  \begin{align*}
    \sum_{k=1}^{M_\delta} \delta^s\cdot \delta^{-1} \sup_{u,v\in \widetilde I_k} \abs{f(u)-f(v)}
    &\leq \sum_{k=1}^{M_\delta} \delta^{\sigma+s-1} \,\abs{\widetilde x_k - t}^{-s'} \\
    &\leq \delta^{\sigma+s-1} \sum_{k=1}^{M_\delta}  \pthb{ C_\gamma n^{1-\gamma}+k\delta }^{-s'}.
  \end{align*}
  The sum can upper bounded by
  \begin{align*}
    \int_0^{M_\delta} \pthb{C_\gamma n^{1-\gamma} + x\delta }^{-s'} \dt x
    &\leq \delta^{-1}\int_0^{\rho} \pthb{C_\gamma n^{1-\gamma} + v }^{-s'} \dt v \\
    &\leq c_2 \,\delta^{-1} \cdot n^{(\gamma-1)(s'-1)} \\
    &\leq c_3 \,\delta^{-1} \cdot \delta^{-(\gamma-1)(s'-1)/\gamma}.
  \end{align*}
  Combining the previous estimate, we obtain
  \[
    \sum_{k=1}^{M_\delta} \delta^s\cdot \delta^{-1} \sup_{u,v\in \widetilde I_k} \abs{f(u)-f(v)}
    \leq c_4 \,\delta^{\sigma+s-2-(\gamma-1)(s'-1)/\gamma}.
  \]
  Hence, the latter converges as $\delta\rightarrow 0$ when $s > 2-\sigma+(\gamma-1)(s'-1)/\gamma$. Therefore, combining the two conditions we have obtained, the Hausdorff dimension satisfies
  \[
    \dimH \gr(f,\ivff{t-\rho,t+\rho}) \leq \min\brcb{ 1 -\sigma +s'(1-\gamma^{-1}) + \gamma^{-1}, 2-\sigma+(\gamma-1)(s'-1)/\gamma },
  \]
  for all $\gamma > 1$ and $s'\in\R$. Considering the limit $\gamma\rightarrow 1$, we obtain $\dimH \gr(f,\ivff{t-\rho,t+\rho}) \leq 2 - \sigma$, which implies the expected inequality.
\end{proof}

We note that the two new upper bounds presented in the previous Lemmas are consistent with the original one, since $\widetilde\alpha_{f,t} = \sigma_{f,t}(0) \leq \sigma_{f,t}(1) \leq \sigma_{f,t}(+\infty)$. Furthermore, it proves to be the optimal ones for a classic deterministic function called the \emph{chirp function}: $f:x\mapsto\abs{x}^\alpha\sin(\abs{x}^{-\beta})$, with $\alpha,\beta > 0$. As proved by \citet{Echelard-2007}, its 2-microlocal frontier at $0$ is equal to
\[
  \forall s'\in\R;\quad \sigma_{f,0}(s') = \frac{s'+\alpha}{1+\beta}.
\]
It is an easy calculus to check that the local Hausdorff and Box dimensions at zero are given by
\[
  \dimHt{0}\gr(f) = 1 \quad\text{and}\quad \dimBt{0}\gr(f) = 2 - \pthB{\frac{1+\alpha}{1+\beta}}\wedge 1,
\]
Hence, they are equal to the upper bounds presented in Lemmas~\ref{lemma:2ml_dimB} and \ref{lemma:2ml_dimH}, on the contrary to the usual Hölder estimate $2 - \alpha / (1+\beta)$ which is not optimal.\vsp

To end this section, we prove two useful lemmas related to the parabolic Hausdorff dimension. The first one constitutes an extension of an inequality proved by \citet{Peres.Sousi-2013} (Lemma $2.1$).
\begin{lemma}  \label{lemma:dim_para_bounds}
  For every $A\subset\R^2$ and all $H_1 > H_2 > 0$,
  \[
    \dimH\pthb{A \,; \varrho_{H_2} } + \frac{1}{H_2} - \frac{1}{H_1}
    \hex\leq\hex \dimH\pthb{A \,; \varrho_{H_1} }
    \hex\leq\hex 1 + \frac{H_2}{H_1} \pthb{ \dimH\pthb{A \,; \varrho_{H_2} } - 1 }.
  \]
  As a consequence, if $H(\cdot)$ is a positive continuous function, the map $t\mapsto \dimH\pthb{A \,; \varrho_{H(t)} }$ is also continuous.
\end{lemma}
\begin{proof}
  Suppose $A\subset\R^2$ and $H_1 > H_2$. Let us first prove the upper bound on $\dimH\pthb{A \,; \varrho_{H_1} }$. For any $\gamma > \dimH\pthb{A \,; \varrho_{H_2} }$, there exists a $\varrho_{H_2}$-cover $(O_i)_{i\in\N}$ of $A$ such that
  \[
    \sum_{i=0}^\infty \delta_i^\gamma < \infty \hem\text{where}\hem \delta_i \eqdef \diam\pthb{O_i\,;\varrho_{H_2}}\hex\forall i\in\N.
  \]
  We denote by $\rho_i$ the quantity $\delta_i^{H_1/H_2}$. Without any loss of generality, we may assume that for every $i\in\N$, $O_i$ is a rectangle of size $\delta_i^{1/H_2} \times \delta_i$ (as a simple consequence of the definition of $\varrho_{H_2}$). We want to construct a $\varrho_{H_1}$-cover of $A$ from $(O_i)_{i\in\N}$. Since $\rho_i^{1/H_1} = \delta_i^{1/H_2}$, we need to split $O_i$ along the space axis to obtain $\varrho_{H_1}$-balls. More precisely, there will be at most $2\delta_i/\rho_i = 2\delta_i^{1-H_1/H_2}$ resulting $\varrho_{H_1}$-balls of diameter $\rho_i$. Hence, we obtain a $\varrho_{H_1}$-cover $(V_k)_{k\in\N}$ of $A$ such that
  \[
    \sum_{k=0}^\infty \diam\pthb{V_k\,;\varrho_{H_1}}^s \leq 2\sum_{i=0}^\infty \delta_i^{1-H_1/H_2} \cdot\delta_i^{sH_1/H_2}.
  \]
  The last sum is finite when $1 + H_1/H_2 (s-1) > \gamma$, i.e. when $s > 1 + (\gamma-1)H_2/H_1$. Considering $\gamma\rightarrow\dimH\pthb{A \,; \varrho_{H_2} }$, we obtain the expected upper bound.

  The lower bound is obtained using a similar reasoning. For any $s > \dimH\pthb{A \,; \varrho_{H_1} }$, there exists a $\varrho_{H_1}$-cover $(V_k)_{k\in\N}$ of $A$ such that
  \[
    \sum_{k=0}^\infty \rho_k^s < \infty \hem\text{where}\hem \rho_k \eqdef \diam\pthb{V_k\,;\varrho_{H_1}}\hex\forall k\in\N.
  \]
  Similarly, we may assume that $V_k$ is a rectangle of size $\rho_k^{1/H_1} \times \rho_k$ for every $k\in\N$. As $H_1 > H_2$, we note that $\rho^{1/H_2} < \rho_k^{1/H_1}$. Hence, to obtain a $\varrho_{H_2}$-cover of $A$, we split the previous balls along the time axis, in at most $2\rho_k^{1/H_1-1/H_2}$ sub-balls. Therefore, we obtain a $\varrho_{H_2}$-cover $(O_i)_{i\in\N}$ of $A$ such that
  \[
    \sum_{i=0}^\infty \diam\pthb{O_i\,;\varrho_{H_2}}^\gamma \leq 2\sum_{k=0}^\infty \rho_k^{1/H_1-1/H_2} \cdot \rho_k^{\gamma}.
  \]
  This sum is finite when $\gamma + 1/H_1-1/H_2 > s$, inducing the lower bound on $\dimH\pthb{A \,; \varrho_{H_1} }$.

  The continuity of the map $t\mapsto \dimH\pthb{A \,; \varrho_{H(t)} }$ is straightforward as we observe that lower and upper bounds converge to $\dimH\pthb{A \,; \varrho_{H_1} }$ when $H_2\rightarrow H_1$.
\end{proof}

The previous lemma helps us to extend the property of upper semi-continuity to the local parabolic Hausdorff dimension.
\begin{lemma}  \label{lemma:dim_para_usc}
  Suppose $A\subset\R^2$ and $H(\cdot)$ is a positive continuous function. Then, the map $t\mapsto\dimHt{t}\pthb{A \,; \varrho_{H(t)} }$ is upper semi-continuous, i.e. for all $t\in\R$
  \[
    \dimHt{t}\pthb{A \,; \varrho_{H(t)} } \geq \limsup_{s\rightarrow t} \,\dimHt{s}\pthb{A \,; \varrho_{H(s)} }.
  \]
\end{lemma}
\begin{proof}
  For any $t\in\R$, the local parabolic Hausdorff dimension satisfies
  \[
    \dimHt{t}\pthb{A \,; \varrho_{H(t)} } \geq \limsup_{s\rightarrow t} \,\dimHt{s}\pthb{A \,; \varrho_{H(t)} }.
  \]
  This inequality is straightforward extension of the usual one on the Hausdorff dimension. Let $\delta > 0$. Owing to the uniform bounds obtained in Lemma~\ref{lemma:dim_para_bounds}, there exists $c_\delta>0$ such that for all $s\in B(t,\delta)$ and all $\rho>0$,
  \[
    \absb{ \dimH\pthb{A\cap B(s,\rho) \,; \varrho_{H(t)} } - \dimH\pthb{A\cap B(s,\rho) \,; \varrho_{H(s)} } } \leq c_\delta\abs{ H(s)-H(t) }.
  \]
  Hence, considering the limit $\rho\rightarrow 0$, we obtain
  $\abs{ \dimHt{s}\pthb{A \,; \varrho_{H(t)} } - \dimHt{s}\pthb{A \,; \varrho_{H(s)} } } \leq c_\delta\abs{ H(s)-H(t) }$,
  proving, jointly with the first inequality, the upper semi-continuity of the map $t\mapsto\dimHt{t}\pthb{A \,; \varrho_{H(t)} }$.
\end{proof}



\begin{thebibliography}{38}
\providecommand{\natexlab}[1]{#1}
\providecommand{\url}[1]{\texttt{#1}}
\expandafter\ifx\csname urlstyle\endcsname\relax
  \providecommand{\doi}[1]{doi: #1}\else
  \providecommand{\doi}{doi: \begingroup \urlstyle{rm}\Url}\fi

\bibitem[Ayache(2013)]{Ayache-2013}
A.~Ayache.
\newblock Continuous {G}aussian multifractional processes with random pointwise
  {H}\"older regularity.
\newblock \emph{J. Theoret. Probab.}, 26\penalty0 (1):\penalty0 72--93, 2013.

\bibitem[Ayache and Taqqu(2005)]{Ayache.Taqqu-2005}
A.~Ayache and M.~S. Taqqu.
\newblock Multifractional processes with random exponent.
\newblock \emph{Publ. Mat.}, 49\penalty0 (2):\penalty0 459--486, 2005.

\bibitem[Ayache et~al.(2000)Ayache, Cohen, and {L\'evy
  V\'ehel}]{Ayache.Cohen.ea-2000}
A.~Ayache, S.~Cohen, and J.~{L\'evy V\'ehel}.
\newblock The covariance structure of multifractional brownian motion, with
  application to long range dependence.
\newblock In \emph{Proceedings of the Acoustics, Speech, and Signal Processing,
  2000}, volume~6, pages 3810--3813, Washington, DC, USA, 2000. IEEE Computer
  Society.

\bibitem[Ayache et~al.(2011)Ayache, Shieh, and Xiao]{Ayache.Shieh.ea-2011}
A.~Ayache, N.-R. Shieh, and Y.~Xiao.
\newblock Multiparameter multifractional brownian motion: local nondeterminism
  and joint continuity of the local times.
\newblock \emph{Ann. Inst. H. Poincar\'e Probab. Statist}, 2011.

\bibitem[Balan\c{c}a(2013)]{Balanca-2013}
P.~Balan\c{c}a.
\newblock Fine regularity of {L}\'evy processes and linear (multi)fractional
  stable motion.
\newblock \emph{Preprint}, 2013.
\newblock \href{http://arxiv.org/abs/1302.3140}{\texttt{arXiv:1302.3140}}.

\bibitem[Benassi et~al.(1997)Benassi, Jaffard, and
  Roux]{Benassi.Jaffard.ea-1997}
A.~Benassi, S.~Jaffard, and D.~Roux.
\newblock Elliptic {G}aussian random processes.
\newblock \emph{Rev. Mat. Iberoamericana}, 13\penalty0 (1):\penalty0 19--90,
  1997.

\bibitem[Berman(1973)]{Berman-1973}
S.~M. Berman.
\newblock Local nondeterminism and local times of {G}aussian processes.
\newblock \emph{Indiana Univ. Math. J.}, 23:\penalty0 69--94, 1973.

\bibitem[Bony(1986)]{Bony-1986}
J.-M. Bony.
\newblock Second microlocalization and propagation of singularities for
  semilinear hyperbolic equations.
\newblock In \emph{Hyperbolic equations and related topics ({K}atata/{K}yoto,
  1984)}, pages 11--49. Academic Press, Boston, MA, 1986.

\bibitem[Boufoussi et~al.(2006)Boufoussi, Dozzi, and
  Guerbaz]{Boufoussi.Dozzi.ea-2006}
B.~Boufoussi, M.~Dozzi, and R.~Guerbaz.
\newblock On the local time of multifractional {B}rownian motion.
\newblock \emph{Stochastics}, 78\penalty0 (1):\penalty0 33--49, 2006.

\bibitem[Boufoussi et~al.(2007)Boufoussi, Dozzi, and
  Guerbaz]{Boufoussi.Dozzi.ea-2007}
B.~Boufoussi, M.~Dozzi, and R.~Guerbaz.
\newblock Sample path properties of the local time of multifractional
  {B}rownian motion.
\newblock \emph{Bernoulli}, 13\penalty0 (3):\penalty0 849--867, 2007.

\bibitem[Charmoy et~al.(2012)Charmoy, Peres, and Sousi]{Charmoy.Peres.ea-2012}
P.~H.~A. Charmoy, Y.~Peres, and P.~Sousi.
\newblock Minkowski dimension of brownian motion with drift.
\newblock 2012.

\bibitem[Echelard(2007)]{Echelard-2007}
A.~Echelard.
\newblock \emph{Analyse 2-microlocale et application au d{\'e}bruitage}.
\newblock PhD thesis, Universit{\'e} de Nantes, 2007.
\newblock \url{http://tel.archives-ouvertes.fr/tel-00283008/fr/}.

\bibitem[Falconer(2003)]{Falconer-2003}
K.~Falconer.
\newblock \emph{Fractal geometry}.
\newblock John Wiley \& Sons Inc., Hoboken, NJ, second edition, 2003.
\newblock Mathematical foundations and applications.

\bibitem[Geman and Horowitz(1980)]{Geman.Horowitz-1980}
D.~Geman and J.~Horowitz.
\newblock Occupation densities.
\newblock \emph{Ann. Probab.}, 8\penalty0 (1):\penalty0 1--67, 1980.

\bibitem[Herbin(2006)]{Herbin-2006}
E.~Herbin.
\newblock From {$N$} parameter fractional {B}rownian motions to {$N$} parameter
  multifractional {B}rownian motions.
\newblock \emph{Rocky Mountain J. Math.}, 36\penalty0 (4):\penalty0 1249--1284,
  2006.

\bibitem[Herbin and L{\'e}vy~V{\'e}hel(2009)]{Herbin.LevyVehel-2009}
E.~Herbin and J.~L{\'e}vy~V{\'e}hel.
\newblock Stochastic 2-microlocal analysis.
\newblock \emph{Stochastic Process. Appl.}, 119\penalty0 (7):\penalty0
  2277--2311, 2009.

\bibitem[Herbin et~al.(2012)Herbin, Arras, and Barruel]{Herbin.Arras.ea-2012}
E.~Herbin, B.~Arras, and G.~Barruel.
\newblock {From almost sure local regularity to almost sure Hausdorff dimension
  for Gaussian fields}.
\newblock \emph{Preprint}, 2012.

\bibitem[Jaffard(1991)]{Jaffard-1991}
S.~Jaffard.
\newblock Pointwise smoothness, two-microlocalization and wavelet coefficients.
\newblock \emph{Publ. Mat.}, 35\penalty0 (1):\penalty0 155--168, 1991.
\newblock Conference on Mathematical Analysis (El Escorial, 1989).

\bibitem[Kahane(1985)]{Kahane-1985}
J.-P. Kahane.
\newblock \emph{Some random series of functions}, volume~5 of \emph{Cambridge
  Studies in Advanced Mathematics}.
\newblock Cambridge University Press, Cambridge, second edition, 1985.

\bibitem[Khoshnevisan and Xiao(2012)]{Khoshnevisan.Xiao-2012}
D.~Khoshnevisan and Y.~Xiao.
\newblock Brownian motion and thermal capacity.
\newblock \emph{Preprint}, 2012.

\bibitem[Khoshnevisan et~al.(2006)Khoshnevisan, Wu, and
  Xiao]{Khoshnevisan.Wu.ea-2006}
D.~Khoshnevisan, D.~Wu, and Y.~Xiao.
\newblock Sectorial local non-determinism and the geometry of the {B}rownian
  sheet.
\newblock \emph{Electron. J. Probab.}, 11:\penalty0 no. 32, 817--843, 2006.

\bibitem[Kolwankar and {L\'evy V\'ehel}(2002)]{Kolwankar.LevyVehel-2002}
K.~M. Kolwankar and J.~{L\'evy V\'ehel}.
\newblock A time domain characterization of the fine local regularity of
  functions.
\newblock \emph{J. Fourier Anal. Appl.}, 8\penalty0 (4):\penalty0 319--334,
  2002.

\bibitem[McMullen(1984)]{McMullen-1984}
C.~McMullen.
\newblock The {H}ausdorff dimension of general {S}ierpi\'nski carpets.
\newblock \emph{Nagoya Math. J.}, 96:\penalty0 1--9, 1984.

\bibitem[Meerschaert et~al.(2008)Meerschaert, Wu, and
  Xiao]{Meerschaert.Wu.ea-2008}
M.~Meerschaert, D.~Wu, and Y.~Xiao.
\newblock Local times of multifractional {B}rownian sheets.
\newblock \emph{Bernoulli}, 14\penalty0 (3):\penalty0 865--898, 2008.

\bibitem[Meyer(1998)]{Meyer-1998}
Y.~Meyer.
\newblock \emph{Wavelets, vibrations and scalings}, volume~9 of \emph{CRM
  Monograph Series}.
\newblock American Mathematical Society, Providence, RI, 1998.

\bibitem[Monrad and Pitt(1987)]{Monrad.Pitt-1987}
D.~Monrad and L.~D. Pitt.
\newblock Local nondeterminism and {H}ausdorff dimension.
\newblock In \emph{Seminar on stochastic processes, 1986 ({C}harlottesville,
  {V}a., 1986)}, volume~13 of \emph{Progr. Probab. Statist.}, pages 163--189.
  Birkh\"auser Boston, Boston, MA, 1987.

\bibitem[Mountford(1989)]{Mountford-1989a}
T.~S. Mountford.
\newblock Uniform dimension results for the {B}rownian sheet.
\newblock \emph{Ann. Probab.}, 17\penalty0 (4):\penalty0 1454--1462, 1989.

\bibitem[Papanicolaou and S{\o}lna(2003)]{Papanicolaou.Solna-2003}
G.~C. Papanicolaou and K.~S{\o}lna.
\newblock Wavelet based estimation of local {K}olmogorov turbulence.
\newblock In \emph{Theory and applications of long-range dependence}, pages
  473--505. Birkh\"auser Boston, Boston, MA, 2003.

\bibitem[Peltier and {L\'evy V\'ehel}(1995)]{Peltier.LevyVehel-1995}
R.~F. Peltier and J.~{L\'evy V\'ehel}.
\newblock Multifractional brownian motion: Definition and preliminary results.
\newblock \emph{Rapport de recherche INRIA}, \penalty0 (2645):\penalty0 1--39,
  1995.

\bibitem[Peres and Sousi(2012)]{Peres.Sousi-2012}
Y.~Peres and P.~Sousi.
\newblock Brownian motion with variable drift: 0-1 laws, hitting probabilities
  and {H}ausdorff dimension.
\newblock \emph{Math. Proc. Cambridge Philos. Soc.}, 153\penalty0 (2):\penalty0
  215--234, 2012.

\bibitem[Peres and Sousi(2013)]{Peres.Sousi-2013}
Y.~Peres and P.~Sousi.
\newblock Dimension of fractional brownian motion with variable drift.
\newblock 2013.

\bibitem[Picard(2011)]{Picard-2011}
J.~Picard.
\newblock {Representation formulae for the fractional Brownian motion}.
\newblock \emph{S{\'e}minaire de Probabilit{\'e}s}, XLIII:\penalty0 3--70,
  2011.

\bibitem[Samko et~al.(1993)Samko, Kilbas, and Marichev]{Samko.Kilbas.ea-1993}
S.~G. Samko, A.~A. Kilbas, and O.~I. Marichev.
\newblock \emph{Fractional integrals and derivatives}.
\newblock Gordon and Breach Science Publishers, Yverdon, 1993.

\bibitem[Seuret and {L\'evy V\'ehel}(2003)]{Seuret.LevyVehel-2003}
S.~Seuret and J.~{L\'evy V\'ehel}.
\newblock A time domain characterization of 2-microlocal spaces.
\newblock \emph{J. Fourier Anal. Appl.}, 9\penalty0 (5):\penalty0 473--495,
  2003.

\bibitem[Stoev and Taqqu(2006)]{Stoev.Taqqu-2006}
S.~A. Stoev and M.~S. Taqqu.
\newblock How rich is the class of multifractional {B}rownian motions?
\newblock \emph{Stochastic Process. Appl.}, 116\penalty0 (2):\penalty0
  200--221, 2006.

\bibitem[Takashima(1989)]{Takashima-1989}
K.~Takashima.
\newblock Sample path properties of ergodic self-similar processes.
\newblock \emph{Osaka J. Math.}, 26\penalty0 (1):\penalty0 159--189, 1989.

\bibitem[Taylor and Watson(1985)]{Taylor.Watson-1985}
S.~J. Taylor and N.~A. Watson.
\newblock A {H}ausdorff measure classification of polar sets for the heat
  equation.
\newblock \emph{Math. Proc. Cambridge Philos. Soc.}, 97\penalty0 (2):\penalty0
  325--344, 1985.

\bibitem[Wu and Xiao(2007)]{Wu.Xiao-2007a}
D.~S. Wu and Y.~M. Xiao.
\newblock Dimensional properties of fractional {B}rownian motion.
\newblock \emph{Acta Math. Sin. (Engl. Ser.)}, 23\penalty0 (4):\penalty0
  613--622, 2007.

\end{thebibliography}

\end{document}